\documentclass[11pt,a4paper,reqno]{amsart}
\usepackage[utf8]{inputenc}
\usepackage{amssymb,amsmath,amsthm,mathrsfs}
\usepackage[abbrev,bibtex-style]{amsrefs}
\usepackage[hidelinks]{hyperref}
\usepackage{pgfplots}
\pgfdeclarelayer{bg}
\pgfsetlayers{bg,main}
\usepackage{subcaption}

\newcommand{\real}{\mathbb{R}}
\newcommand{\nat}{\mathbb{N}}

\newcommand{\nulld}{\hat{d}}
\newcommand{\nullL}{\hat{L}}
\newcommand{\nullR}{\hat{R}^+}
\newcommand{\J}{J^+}
\newcommand{\K}{K^+}
\newcommand{\g}{\bar{g}}

\theoremstyle{plain}
\newtheorem{thm}{Theorem}[section]
\newtheorem{cor}[thm]{Corollary}
\newtheorem{lem}[thm]{Lemma}
\newtheorem{prop}[thm]{Proposition}

\theoremstyle{definition}
\newtheorem{defn}[thm]{Definition}
\newtheorem{ex}[thm]{Example}
\newtheorem{rem}[thm]{Remark}

\title[Global hyperbolicity and the null distance]{Global hyperbolicity\\ through the eyes of the null distance}

\author[A.~Burtscher]{Annegret Burtscher$^\dagger$}
\author[L.~Garc\'ia-Heveling]{Leonardo Garc\'ia-Heveling$^{\dagger,*}$}
\thanks{$^\dagger$Department of Mathematics, IMAPP, Radboud University, PO Box 9010, 
6500 GL Nijmegen, The Netherlands, \textit{Email}: \texttt{burtscher@math.ru.nl}.\\
\indent $^*$Current address: Fachbereich Mathematik, Universit\"at Hamburg, Bundesstra\ss e 55, 20146 Hamburg, Germany. \textit{Email}: {\tt leonardo.garcia@uni-hamburg.de}.\\
{\it Acknowledgements:} \textnormal{This manuscript was completed during an extended research stay at the Fields Institute for Research in Mathematical Sciences in Toronto, in connection with the thematic program ``Nonsmooth Riemannian and Lorentzian Geometry" that took place in the Fall of 2022. Both authors gratefully acknowledge funding by the Fields Institute during their stay. AB's research is in part also supported by the Dutch Research Council (NWO), Project number VI.Veni.192.208.} We thank all referees for their valuable comments.}

\subjclass[2020]{53C50 (primary), 
53C23, 
53B30, 
83C05
}

\keywords{Lorentzian geometry, metric geometry, general relativity, causality theory, completeness, spacetimes, time functions}

\begin{document}

\maketitle

\begin{abstract}
 No Hopf--Rinow Theorem is possible in Lorentzian Geometry. None\-theless, we prove that a spacetime is globally hyperbolic if and only if it is metrically complete with respect to the null distance of a time function. Our approach is based on the observation that null distances behave particularly well for weak temporal functions in terms of regularity and causality. Specifically, we also show that the null distances of Cauchy temporal functions and regular cosmological time functions encode causality globally.
\end{abstract}

\section{Introduction}

The notion of global hyperbolicity was introduced by Leray~\cite{Ler} in 1952 to prove the uniqueness of solutions for hyperbolic partial differential equations. Shortly thereafter, global hyperbolicity entered the field of General Relativity through the proof of the global well-posedness of the Einstein equations of Choquet--Bruhat and Geroch~\cites{CB,CBGe} and the Singularity Theorems of Penrose and Hawking~\cites{HaEl,Pen}. Via the topological splitting result of Geroch~\cite{Ger} globally hyperbolic spacetimes manifestly settled in Lorentzian Geometry in 1970.

Spacetimes are time-oriented Lorentzian manifolds $(M,g)$. They are the geometric objects needed for formulating gravitation in General Relativity. Throughout this manuscript, we use the sign convention $(-,+,\ldots,+)$ for $g$ and assume that it is smooth (although $C^2$ is sufficient, and occasionally less). Global hyperbolicity establishes a deep link between the topology of $M$ and the causal structure induced by the metric tensor $g$. The causal structure is induced on $M$ by causal curves, i.e., locally Lipschitz (with respect to any Riemannian metric \cite{Chr}*{Sec.\ 2.3}, \cite{Bur}*{Thm.\ 4.5}) curves $\gamma$ with $g(\dot\gamma,\dot\gamma)\leq 0$, as follows. If $q$ can be reached by a future-directed causal curve from $p$ we say that $q$ is in the causal future of $p$, and write $q \in J^+(p)$ (dually for the causal past $J^-(p)$) or $(p,q) \in \J$.
Leray's original definition of global hyperbolicity was based on the $C^0$-compactness of the set of causal curves between any two points in a spacetime. The modern definition of global hyperbolicity requires compactness of causal diamonds $J^+(p) \cap J^-(q)$ akin to the Heine--Borel property for complete Riemannian manifolds (see \cite{HaEl}*{Sec.\ 6.6} and \cite{BeSa2}).

\begin{defn}\label{def:gh}
 A spacetime $(M,g)$ is called \emph{globally hyperbolic} if it is \emph{causal} (there is no closed causal curve) and all causal diamonds $J^+(p) \cap J^-(q)$, $p,q \in M$, are compact.
\end{defn}

\noindent If $(M,g)$ is a noncompact spacetime of dimension greater than $2$ then the causal condition can be dropped \cite{HoMi}.

A landmark result concerning global hyperbolicity is Geroch's Topological Splitting Theorem~\cite{Ger}, later promoted to a smooth orthogonal splitting by Bernal and S\'anchez~\cite{BeSa}.
It states that any globally hyperbolic spacetime $(M,g)$ admits a \emph{Cauchy orthogonal splitting}, i.e., an isometry
\begin{equation} \label{eq:GHsplitting}
    (M,g) \cong (\real \times \Sigma, -\alpha d\tau^2 + \g_\tau),
\end{equation}
for $\alpha \colon \real \times \Sigma \to (0,\infty)$ a smooth function and $(\g_\tau)_\tau$ identifiable with a family of Riemannian metrics on the Cauchy slices $\{\tau\}\times\Sigma$, smoothly varying in $\tau$. The proof of this splitting result is rooted in the construction of a suitable time function $\tau$.

\begin{defn}
 Let $(M,g)$ be a spacetime. A continuous function $\tau \colon M \to \real$ is said to be a \emph{time function} if
 \[
  q \in J^+(p)\setminus\{p\} \Longrightarrow \tau(p) < \tau(q).
 \]
\end{defn}

Every \emph{stably causal} spacetime admits a (non-unique) time function \cites{Haw,Min1,Min}. The class of globally hyperbolic spacetimes can conveniently be characterized by the special type of time functions they admit.

\begin{thm}[Geroch~\cite{Ger}, Bernal--S\'anchez~\cite{BeSa}]\label{thm:Cauchytime}
 A spacetime $(M,g)$ is globally hyperbolic if and only if there exists a \emph{Cauchy time function} $\tau$ on $M$, meaning that each of its level sets $\tau^{-1}(s)$, $s\in\real$, is a \emph{Cauchy surface}, i.e., intersected (exactly once) by every inextendible causal curve.
\end{thm}

From the compactness condition in Definition~\ref{def:gh} it follows by a result obtained independently by Avez~\cite{Ave} and Seifert~\cite{Sei} that there exists a length-maximizing geodesic between any two causally related points (length-maximizing with respect to the Lorentzian distance, which is then also finite-valued and continuous~\cite{BEE}*{Ch.\ 4}). In that sense globally hyperbolic spacetimes again resemble complete Riemannian manifolds. But here the analogy ends. The Lorentzian distance is far from inducing a metric space structure. Even more troubled is the relationship with geodesic completeness. Neither does global hyperbolicity imply geodesic completeness (the famous Penrose Singularity Theorem~\cite{Pen} actually shows incompleteness under additional curvature bounds) nor the other way round (anti-de Sitter space). In both cases these are actually features rather than bugs of Lorentzian manifolds, and physically highly desired, for instance, for the mathematical existence of black holes. Nonetheless, even the physically undesired assumption of compactness does not guarantee geodesic completeness (Clifton--Pohl torus). Altogether these properties render any Hopf--Rinow type statement for spacetimes virtually a lost case (see the early works of Busemann~\cite{Bus}, Beem~\cite{Beem}; and \cites{ChYa,BeEr,Har3,Erk} for work on completeness of spacelike submanifolds). We reopen the case and characterize global hyperbolicity in an entirely new way.

\begin{thm}\label{mainthm}
 A spacetime $(M,g)$ is globally hyperbolic if and only if there exists a time function $\tau$ such that $(M,\nulld_\tau)$ is a complete metric space.
\end{thm}

Here $\nulld_\tau$ is the null distance of Sormani and Vega~\cite{SoVe}, defined in 2016 with the purpose of studying geometric stability problems in General Relativity by means of a metric (measure) convergence theory, and to develop robust tools for spacetimes of low regularity (partly already realized in \cites{SoVe,AlBu,SaSo,KuSt,BuGH,Ve}). The $\hat{\;}$ in $\nulld_\tau$ indicates the dependence on the causal cone structure and $\tau$ the link to the time function.

\begin{defn}\label{def:nulld}
 Let $(M,g)$ be a spacetime with time function $\tau$. A \emph{piecewise causal path} $\beta \colon [a,b]\to M$ is given by a partition $a =s_0 < s_1 < \ldots s_{k-1} < s_k = b$ on which each restriction $\beta|_{[s_{i-1},s_i]}$ is either a future- or past-directed causal curve. The \emph{null length} of $\beta$ is given by
 \[
  \nullL_\tau (\beta) := \sum_{i=1}^k |\tau(\beta(s_{i}))-\tau(\beta({s_{i-1}}))|.
 \]
 The \emph{null distance} between two points $p,q \in M$ is
 \[
  \nulld_\tau(p,q) := \inf \{ \nullL_\tau(\beta) \mid \beta \text{ piecewise causal path between $p$ and $q$} \}.
 \]
\end{defn}

Clearly, $\nulld_\tau$ is symmetric and satisfies the triangle inequality, but positive definiteness does not hold for all $\tau$. For locally anti-Lipschitz time functions one indeed obtains a conformally invariant length-metric space $(M,\nulld_\tau)$ that induces the manifold topology (see \cite{SoVe}*{Thm.\ 4.6} and \cite{AlBu}*{Thm.\ 1.1}). Note that the relevance of an anti-Lipschitz condition was recognized already earlier and is important for time functions also in other contexts \cites{AGH,CGM}. 
Throughout most of this manuscript we will assume slightly more, namely that $\tau$ is a (weak) temporal function. 

\begin{defn}\label{LaLdef}
 Let $(M,g)$ be a spacetime with time function $\tau \colon M \to \real$. Let $h$ be any Riemannian metric on $M$ and $d_h$ the associated distance function. If for every point $x$ there exists a neighborhood $U$ of $x$ and $C\geq 1$ such that
  \begin{align}\label{LaL}
   (p,q) \in J^+ \cap (U \times U)  \Longrightarrow \frac{1}{C} d_h (p,q) \leq \tau(q) - \tau(p) \leq C d_h(p,q),
 \end{align}
 then we say that $\tau$ is a \emph{weak temporal function}. If $\tau$ satisfies only the first $\leq$ in \eqref{LaL} it is called \emph{locally anti-Lipschitz}, if only the second $\leq$ it is called \emph{locally Lipschitz}.
\end{defn}

Based on our extension of \eqref{LaL} to an entire open set in Section~\ref{sec:equiv} we show that weak temporal functions are indeed locally Lipschitz (in the usual sense) and have a timelike gradient almost everywhere. The standard (smooth) temporal functions, and also regular cosmological time functions \`{a} la Andersson--Galloway--Howard~\cite{AGH} and Wald--Yip~\cite{WaYi}, are weak temporal functions. In fact, working with temporal functions is no restriction since every smooth spacetime that admits a time function also admits a temporal function \cite{BeSa}*{Thm.\ 1.2}. The true advantage of temporal functions over other time functions is the orthogonal decomposition $g = -\alpha d\tau^2 + \g_\tau$ (although no product splitting of the manifold, see \cite{MuSa}*{Lem.\ 3.5}). This property is key in the proof of the subsequent results. Time functions are, however, also a useful tool in weaker nonsmooth geometric settings (see, for instance, \cites{BuGH,BDGSS,KuSt}) and also regular cosmological time functions are, in general, not smooth even on smooth spacetimes. Therefore, we decided to prove our results for the optimal regularity class.

\begin{thm} \label{thm:nonsmooth}
Let $(M,g)$ be a spacetime, $\tau \colon M \to \real$ be a weak temporal function, and $h$ be a Riemannian metric on $M$. Then, for each compact set $K \subseteq M$, there exists a constant $C \geq 1$ such that for all $p,q \in K$,
\begin{equation} \label{eq:equivnonsmooth}
    \frac{1}{C} d_{h}(p,q) \leq \nulld_\tau(p,q) \leq {C} d_{h}(p,q).
\end{equation}
\end{thm}

Note that the lower bound follows from the locally anti-Lipschitz property of $\tau$, and the upper bound from the corresponding locally Lipschitz bound. Theorem~\ref{thm:nonsmooth} immediately implies that two stably causal spacetimes are metrically equivalent on compact sets in the following sense.

\begin{cor}\label{cor:differentg}
 Let $M$ be a smooth manifold and $g$ and $\tilde g$ be spacetime metrics on $M$ with weak temporal functions $\tau$ and $\tilde\tau$ (and corresponding null distances $\nulld_\tau$ and $\nulld_{\tilde\tau}$), respectively. Then, for each compact set $K \subseteq M$, there exists a constant $C\geq 1$ such that for all $p,q \in K$,
\begin{align*}
  \frac{1}{C} \hat d_{\tau}(p,q) \leq \hat d_{\tilde\tau}(p,q) \leq C \hat d_{\tau}(p,q). \qedhere
 \end{align*}
\end{cor}

Amongst others, the proof of Theorem~\ref{thm:nonsmooth} requires semiglobal techniques to go from open sets to compact sets. On a global scale the situation is even more involved and it is here where globally hyperbolic spacetimes really shine. We prove the following.

\begin{thm} \label{thm:admcausalityglobal}
 Let $(M,g)$ be a globally hyperbolic spacetime and $\tau$ a locally anti-Lipschitz time function such that all nonempty level sets are future (or past) Cauchy. Then the null distance \emph{encodes causality}, that is, for any $p,q \in M$,
 \begin{align}\label{encodingcausality}
  q \in J^+(p) \Longleftrightarrow \nulld_\tau(p,q) = \tau(q) - \tau(p).
 \end{align}
\end{thm}

The $\Longrightarrow$ direction in Theorem~\ref{thm:admcausalityglobal} is trivial. Sormani and Vega~\cite{SoVe}*{Thm.\ 3.25} showed that the converse holds for warped product spacetimes with complete Riemannian fiber and suitable temporal functions. It remained an open problem to determine under which general circumstances causality is encoded. Our Theorem~\ref{thm:admcausalityglobal} provides a sharp answer both in terms of regularity as well as the causality class (see counterexamples in Section~\ref{ssec:causalitycount}). Initially we proved this result for Cauchy temporal functions in Theorem~\ref{thm:causalityglobal}. Independently and simultaneously, Sakovich and Sormani~\cite{SaSo}*{Thm.\ 4.1} have obtained a different global causality encoding result where they allow for general anti-Lipschitz time functions, but require them to be proper. This properness assumption, in fact, implies that the spacetime must be globally hyperbolic with \emph{compact} Cauchy level sets (see Section~\ref{ssec:admcausalityglobal}). Both approaches yield \emph{local} encodement of causality on any stably causal spacetime (see Theorem~\ref{thm:causalitylocal} and \cite{SaSo}*{Thm.\ 1.1}). Upon studying the proofs of \cite{SaSo} we noticed that by combining part of their local  arguments~\cite{SaSo}*{Thm.\ 1.1} with our global proof of Theorem~\ref{thm:causalityglobal} we can obtain Theorem~\ref{thm:admcausalityglobal} which is optimal both in view of regularity as well as causality. It is precisely this optimality, together with the observation that $\tau$ having future (or past) Cauchy level sets is actually sufficient, that allows us to conclude with the following application.

\begin{cor}\label{cor:cosmo}
 Let $(M,g)$ be a spacetime that admits a regular cosmological time function $\tau$. Then $\nulld_\tau$ encodes causality globally.
\end{cor}

The manuscript is structured as follows. In Section~\ref{sec:equiv} we prove Theorem~\ref{thm:nonsmooth} and Corollary~{\ref{cor:differentg}. In the proof we make use of the orthogonal decomposition of the spacetime metric with respect to a temporal function and techniques developed in \cite{Bur}. We also present counterexamples that show that the weak temporal condition cannot be relaxed. In Section~\ref{sec:causality} we prove Theorem~\ref{thm:admcausalityglobal} and provide counterexamples for non-Cauchy locally anti-Lipschitz functions. Nonetheless, a local result for any temporal function on any stably causal spacetime is also obtained. In Section~\ref{sec:complete} we prove Theorem~\ref{mainthm} and show that completely uniform temporal functions (very recently introduced in \cites{BeSu,BeSu2}) guarantee completeness of $(M,\nulld_\tau)$.


\section{Bi-Lipschitz bounds}\label{sec:equiv}

In this section we prove Theorem~\ref{thm:nonsmooth} and Corollary~\ref{cor:differentg} of the introduction.

Since weak temporal functions satisfy $\nulld_\tau(p,q) = \tau(q) - \tau(p)$ for causally related points, the condition \eqref{LaL} can be viewed as a restricted local metric equivalence.
In order to extend this property to a true local metric equivalence we make use of several technical results related to temporal functions obtained in Section~\ref{ssec:wick}, some of which are also used in Section~\ref{sec:causality}. To extend the corresponding local result to weak temporal functions and to compact sets we employ semiglobal techniques similar to those in \cite{Bur}. These final steps of the proof of Theorem~\ref{thm:nonsmooth} are carried out in Section~\ref{ssec:equiv}.

Theorem~\ref{thm:nonsmooth} implies that we can compare the null distances of any two spacetime metrics with entirely different causal cones as stated in Corollary~\ref{cor:differentg} of the introduction.
In particular, the null distance structures with respect to weak temporal functions are equivalent on compact sets of a fixed spacetime.

\begin{cor}\label{equivthm1}
 Let $\tau_1$, $\tau_2$ be two weak temporal functions on a spacetime $(M,g)$, and $\nulld_{\tau_1}$, $\nulld_{\tau_2}$ their associated null distances. Then, for every compact set $K \subseteq M$, there exists a constant $C \geq 1$ such that for all $p,q \in K$,
 \begin{align*}
  \frac{1}{C} \hat d_{\tau_1}(p,q) \leq \hat d_{\tau_2}(p,q) \leq C \hat d_{\tau_1}(p,q). \qedhere
 \end{align*}
\end{cor}

Note that Corollary~\ref{equivthm1} was already announced in \cite{AlBu}*{p.\ 7739} with an alternative direct proof. The advantage of such a proof is that it does not require the use of temporal functions which, for instance, do not exist in theory of Lorentzian length spaces \cite{BuGH} (while weak temporal functions only require local Lipschitz conditions and can still be considered). We therefore provide this short alternative proof of Corollary~\ref{equivthm1} in Section~\ref{ssec:altequivtauproof}.

The question remains how optimal all the results just mentioned are with respect to the regularity class of time functions considered. The lower bound in \eqref{LaL} is the standard local anti-Lipschitz assumption on $\tau$ and needed to even obtain a sensible metric space $(M,\hat d_\tau)$ \cite{SoVe}*{Thm.\ 4.6}. In Section~\ref{ssec:equivcount} we show that the locally Lipschitz assumption (the upper bound in \eqref{LaL}) cannot be dropped either. Besides,the two assumptions together imply that weak temporal functions have a timelike gradient almost everywhere, motivating their name (Section~\ref{ssec:gradient}). Moreover, we show that even under the best circumstances, a general global version of our weakest result presented in this section, namely Corollary~\ref{equivthm1},
cannot be expected. This implies that also Theorem~\ref{thm:nonsmooth} and Corollary~\ref{cor:differentg} in general do not hold globally.


\subsection{Temporal functions and Wick-rotated metrics}\label{ssec:wick}

The aim of this section is to compare the null distance to the distance obtained with respect to the Wick-rotated metric that exists for a temporal function. Recall that a spacetime admits a smooth temporal function whenever it admits a time function by a well-known result of Bernal and S\'anchez \cite{BeSa}*{Thm.\ 1.2}.

\begin{defn}
 Let $(M,g)$ be a spacetime. A \emph{temporal function} is a smooth function $\tau \colon M \to \real$ with past-directed timelike gradient $\nabla\tau$.
\end{defn}

In the local proofs in this section we make use of a weaker splitting result for temporal functions (compare to \eqref{eq:GHsplitting} in the globally hyperbolic case).

\begin{lem}[{M\"uller--S\'anchez \cite{MuSa}*{Lem.\ 3.5}}]
 If a spacetime $(M,g)$ admits a temporal function $\tau$, then the metric $g$ admits an orthogonal decomposition
 \begin{align}\label{ga1}
  g = -\alpha d\tau^2 + \bar g,
 \end{align}
 where $\alpha = |g(\nabla\tau,\nabla\tau)|^{-1} >0$ and $\g$ is a symmetric $2$-tensor which vanishes on $\nabla \tau$ and is positive definite on the complement.
\end{lem}

The temporal function $\tau$ is said to be \emph{steep} if $g(\nabla\tau,\nabla\tau) \leq -1$, and hence $\alpha = |g(\nabla\tau,\nabla\tau)|^{-1} \leq 1$. Not every (even causally simple) spacetime admits a steep temporal function \cite{MuSa}*{Thm.\ 1.1 and Ex.\ 3.3}, but we can always rewrite \eqref{ga1} as $g = -\alpha (d\tau^2 + \tilde{g})$, for $\tilde{g}= \alpha^{-1} \g$. Since the null distance is conformally invariant we assume from now on, without loss of generality, that $g$ is of the form
\begin{align}\label{g}
 g = -d\tau^2 + \g.
\end{align}

Next we perform a standard trick to obtain a Riemannian metric $g_W$ from $g$ and $\tau$. This technique is called Wick-rotation in the physics literature. The \emph{Wick-rotated metric} $g_W$ is given by
 \begin{align}\label{gW}
 g_W := d\tau^2 + \g.
 \end{align}
We denote the associated norm, length functional, and Riemannian distance by $\left\Vert \cdot \right\Vert_W$, $L_W$, and $d_W$ respectively. We proceed to compare $d_W$ and $\nulld_\tau$.
 
\begin{lem} \label{lem:wick}
  Let $(M,g)$ be a spacetime equipped with a smooth temporal function $\tau$, so that $g$ and $g_W$ are given by \eqref{g} and \eqref{gW}, respectively. Let $\beta \colon [a,b] \to M$ be a piecewise causal curve. Then
  \[
  \nullL_\tau (\beta) \leq L_W(\beta) \leq \sqrt{2} \nullL_\tau(\beta),
  \]
  and thus for any $p,q \in M$,
  \begin{align}\label{eq:dWdtau}
   d_W(p,q) \leq \sqrt{2} \nulld_\tau(p,q).
  \end{align}
 \end{lem}
 
 \begin{proof}
 By assumption $\tau$ is smooth and also the tangent vector $\dot\beta$ exists almost everywhere, thus we can write both length functionals in terms of the integrals
 \begin{align*}
     \nullL_\tau (\beta) &= \int_a^b \vert (\tau \circ \beta)'(s) \vert ds, \\
     L_W(\beta) &= \int_a^b \Vert \dot\beta(s) \Vert_W ds = \int_a^b \sqrt{\vert d\tau(\dot \beta(s)) \vert^2 + \Vert \dot\beta(s) \Vert^2_{\g}} ds.
 \end{align*}
 Since $\dot\beta^\tau(s):=(\tau \circ \beta)'(s) = d\tau_{\beta(s)}(\dot\beta(s))$ and the curve $\beta$ is piecewise causal, that is, $g(\dot\beta,\dot\beta) \leq 0$ almost everywhere, we have that
 \begin{equation*}
     \vert \dot\beta^\tau \vert \geq \Vert \dot\beta \Vert_{\g},
 \end{equation*}
 and the inequalities for the lengths follow immediately.
 
 The second inequality for lengths descends to the level of distances, because the class of piecewise causal curves considered (for $\nulld_\tau$) is contained in the class of all locally Lipschitz curves (strictly speaking, $d_W$ is obtained via piecewise smooth curves, but \cite{Bur}*{Cor.\ 3.13} shows that it is the same as the intrinsic metric obtained via the class of absolutely continuous curves, each of which has a locally Lipschitz reparametrization \cite{AGS}*{Lem.\ 1.1.4}).
 \end{proof}
 
 The next lemma proves a reverse inequality to \eqref{eq:dWdtau}, albeit only locally and via a more involved proof (because $d_W$ requires knowledge also about non-causal curves).
 
\begin{lem} \label{lem:wick2}
 Let $(M,g)$ be a spacetime equipped with a smooth temporal function $\tau$, so that $g$ and $g_W$ are given by \eqref{g} and \eqref{gW}, respectively. Then, around every $x \in M$ there is a neighborhood $U$ such that for all $p,q \in U$,
  \begin{equation*}
      \nulld_\tau(p,q) \leq 4 d_W(p,q).
  \end{equation*}
\end{lem}

\begin{proof}
Let $x \in M$ and $\dim M = N+1$. The idea is to use coordinates adapted to the vector field $\nabla\tau$ and the hypersurface $\mathcal{S}_{\tau(x)} :=\tau^{-1}(\tau(x))$ to explicitly construct approximating piecewise causal curves $\beta_n$ in a neighborhood $U$ of $x$. The construction of these coordinates resembles that of the well-known Gaussian normal coordinates, the only difference being that we use the flow of $\nabla \tau$ instead of the normal geodesics to the hypersurface.
 
\medskip
\textbf{Step 1. Construction of $U$ and coordinates.}
The level set $\mathcal{S}_{\tau(x)}$ is a hypersurface in $M$ (since $d\tau\neq 0$), therefore there exists a neighborhood $V \subseteq \mathcal{S}_{\tau(x)}$ of $x$ and a coordinate map $\varphi \colon V \to \varphi(V)\subseteq \real^N$. Let $y \in \varphi(V)\subseteq\real^N$ and let $c_y$ denote the unique integral curve of $-\nabla \tau$ in $M$ with $c_y(0) \in V \subseteq \mathcal{S}_{\tau(x)} \subseteq M$ such that $\varphi(c_y(0)) = y$. Since we have assumed that $d\tau(- \nabla \tau) = - g(\nabla \tau, \nabla \tau) = 1$, we have that
\begin{equation} \label{eq:firstcootau}
 \tau(c_y(t)) = \tau(x) + t.
\end{equation}
The Flow Box Theorem \cite{Lee}*{Thm.\ 2.91} guarantees the existence of a suitable $a>0$ and a neighborhood of $x$ (which we immediately restrict to our chart neighborhood on $\mathcal{S}_{\tau(x)}$ and again denote by $V$) such that the map
 \begin{align*}
  \phi \colon (-a,a) \times \varphi(V) &\to M, \\ (t,y) &\mapsto c_y(t),
\end{align*}
is well-defined. By the Inverse Function Theorem, $\phi$ is a local diffeomorphism, hence one obtains a coordinate system $(t,y)$ on a neighborhood $U$ of $x$ in $M$. At any point $z = (t,y) \in \phi^{-1}(U)\subseteq (-a,a) \times \real^N$, the differential $D_{z}\phi \colon \real \times \real^N \to T_{\phi(z)}M$ maps $(1,0,\ldots,0)$ to $\dot c_y(t) = \nabla \tau |_{\phi(z)}$. Thanks to \eqref{eq:firstcootau}, $D_{z}\phi$ also maps $\{0\} \times \real^N$ to $T_{\phi(z)} \mathcal{S}_{\tau(\phi(z))} \subseteq T_{\phi(z)}M$. Since $T_{\phi(z)} \mathcal{S}_{\tau(\phi(z))}= (\operatorname{span}\nabla\tau|_{\phi(z)})^\perp$, pulling back the metric tensor $g$ with $\phi$, we see that its components in the $(t,y)$ coordinates are
\begin{align}\label{eq:gorth}
g_{00} =-1 && g_{0 i} = 0 \ \text{ for all } \ i \neq 0,
\end{align}
where the subindex $0$ corresponds to the $t$-component and the subindex $i$ to any of the $y$-components. Finally, by further shrinking $U$, we may assume that it is $g_W$-convex.

\medskip
\textbf{Step 2. Construction of piecewise causal curves $\beta_n$ approximating $d_W(p,q)$.}
Let $p,q \in U$ be given. By assumption on $U$, there exists a length-minimizing $g_W$-geodesic $\gamma \colon [0,L] \to U$ from $p$ to $q$, parametrized by arclength. In particular, \[L = L_W(\gamma) = d_W(p,q).\] We write $\gamma = (\gamma^\tau,\gamma^1,...,\gamma^N)$ to denote the coordinate expression of $\gamma$. Consider the sequence $(\beta_n)_n$ of curves given in coordinates by
\begin{equation} \label{eq:betan}
    \beta_n(s) := \left(\gamma^\tau(s) + 3 f_n(s), \gamma^1(s),...,\gamma^N(s)\right),
\end{equation}
for $|f_n|$ sufficiently small so that $\beta_n$ is contained in $U$. Furthermore, we choose $(f_n)_n$ in a way that $\beta_n$ is piecewise causal, that is,
\[
 g(\dot\beta_n,\dot\beta_n) \leq 0 \quad \text{almost everywhere},
\]
and so that $\nullL_\tau(\beta_n) \leq 4 L_W(\gamma)$ for $n$ sufficiently large.

We show that the functions $f_n \colon [0,L] \to \real$, given by
\begin{align*}
 f_0(s) &:= \begin{cases}
            s & \text{for } s \in [0,\frac{L}{2}], \\
            L-s & \text{for } s \in [\frac{L}{2},L],
           \end{cases} \quad \text{and}  \\
 f_n(s) &:= \frac{1}{2} \begin{cases}
                        f_{n-1}(2s) & \text{for } s \in [0,\frac{L}{2}], \\ 
                        f_{n-1}(2s-L) & \text{for } s \in [\frac{L}{2},L],
                    \end{cases}
\end{align*}
as depicted in Figure \ref{fig:wick2a} satisfy all these properties. The curves $\beta_n$ given by \eqref{eq:betan} are shown in Figure \ref{fig:wick2b}. By definition, $\dot\beta_n$ exists almost everywhere, and by \eqref{eq:gorth} it follows that
\begin{align}\label{gbetan1}
 g(\dot\beta_n,\dot\beta_n) &= - |\dot\beta_n^\tau|^2 + \g(\dot\beta_n,\dot\beta_n) \nonumber \\
 &=-|\dot\gamma^\tau_n+3f_n'|^2 + (\g_{ij}\circ\beta_n) \dot\gamma^i \dot\gamma^j.
\end{align}
Since $| f_n'|=1$ almost everywhere, and $|\dot\gamma^\tau_n| \leq g_W(\dot\gamma,\dot\gamma)=1$, it follows that almost everywhere
\begin{align}\label{eq:2first}
 2 \leq 3|f_n'| - |\dot\gamma^\tau_n| \leq |\dot\gamma^\tau_n+3f_n'| \leq 3|f_n'| + |\dot\gamma^\tau_n| \leq 4.
\end{align}
We consider the second term in \eqref{gbetan1}. Since $\gamma$ is a $g_W$-geodesic it is smooth and so are all coordinate components $\dot\gamma^i$ of its tangent vector. Due to compactness of $\gamma$ the $\dot\gamma^i$ must therefore also be bounded, say $|\dot\gamma^i| \leq M$ for some $M>0$. Due to $|f_n(s)| \leq \frac{L}{2^{n+1}} \to 0$ uniformly as $n\to\infty$ it follows that $\beta_n^\tau \to \gamma_n^\tau$ uniformly, and thus $\beta_n \to \gamma$ uniformly in coordinates. The coordinate functions $\g_{ij}\circ\beta_n \to \g_{ij}\circ\gamma$ converge uniformly on $[0,L]$ as $n\to\infty$ as well.
Thus for any $\varepsilon \in (0,\frac{1}{M^2})$ there is an $n_0 \in \nat$ such that for all $n\geq n_0$, for all $s \in [0,L]$,
\begin{align}\label{eq:2second}
 \g_{ij}&(\beta_n(s)) \dot\gamma^i(s) \dot\gamma^j(s) \nonumber \\
 &\leq \g_{ij}(\gamma(s)) \dot\gamma^i(s) \dot\gamma^j(s) + |\g_{ij}(\beta_n(s)) - \g_{ij}(\gamma(s))| |\dot\gamma^i(s)| |\dot\gamma^j(s)| \nonumber \\
 & \leq g_W(\dot\gamma(s),\dot\gamma(s)) + \varepsilon M^2 \leq 2.
\end{align}
Together, \eqref{gbetan1}--\eqref{eq:2second} imply 
\[
 g(\dot\beta_n,\dot\beta_n) \leq -2^2 + 2 \leq 0,
\]
hence all $\beta_n$ for $n\geq n_0$ are piecewise causal.

Using again that \eqref{eq:2first} holds almost everywhere, it follows that for all $n$
\begin{equation*}
    \nullL_\tau (\beta_n) = \int_0^L | (\tau \circ \beta_n)'(s) |ds = \int_0^L | \dot\gamma^\tau(s) + 3f'_n(s) |ds \leq 4 L =4L_W(\gamma),
\end{equation*}
and since $\gamma$ is a length-minimizing $g_W$-geodesic we conclude that
\begin{equation*}
    \nulld_\tau(p,q) \leq 4L = 4 d_W(p,q). \qedhere
\end{equation*}
\end{proof}
 
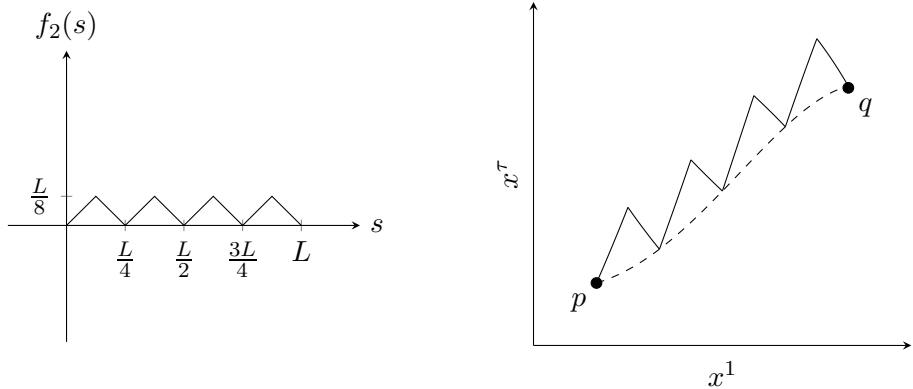
\begin{figure}
    \centering
    \begin{subfigure}{0.5\textwidth}
    \begin{tikzpicture}
    \begin{axis}[xmin=-1,xmax=5,
                 ymin=-2.9,ymax=3,
                 axis x line=middle,
                 axis y line=middle,
                 axis equal image,
                 xlabel={$s$},
                 ylabel={$f_2(s)$},
                 every axis x label/.style={at={(ticklabel* cs:1)},anchor=west,},
                 every axis y label/.style={at={(ticklabel* cs:1)},anchor=south,},
                 xtick={0, 1, 2, 3, 4},
                 xticklabels ={$0$,$\frac{L}{4}$,$\frac{L}{2}$,$\frac{3L}{4}$,$L$},
                 ytick=0.5,
                 yticklabels={$\frac{L}{8}$},
                 scale=0.8]
        \draw foreach \k in {1, ..., 4} {(axis cs: {\k-1},0) -- (axis cs: {\k-0.5},0.5) -- (axis cs: {\k},0)};
        \draw[color=white,very thick] (axis cs: 0,-2) -- (axis cs: 0, -2.9);
    \end{axis}
    \end{tikzpicture}
    \caption{The function $f_2$.}
    \label{fig:wick2a}
    \end{subfigure}%
    \begin{subfigure}{0.5\textwidth}
    \begin{tikzpicture}
    \begin{axis}[xmin=-1,xmax=5,
                 ymin=-1,ymax=4.5,
                 axis x line=bottom,
                 axis y line=left,
                 xlabel={$x^1$},
                 ylabel={$x^\tau$},
                 xtick=\empty,
                 ytick=\empty,
                 axis equal image,
                 ylabel near ticks,
                 xlabel near ticks,
                 scale=0.8]
        \addplot[domain=0:4,samples=100,style=dashed]{0.3*x+0.25*x^2-0.008*x^4};
        \foreach \k in {1, ..., 4} \addplot[domain={\k-1}:{\k-0.5},samples=10]{0.3*x+0.25*x^2-0.008*x^4+2*(x-\k+1)};
        \foreach \k in {1, ..., 4} \addplot[domain={\k-0.5}:{\k},samples=10]{0.3*x+0.25*x^2-0.008*x^4-2*(x-\k)};
        \filldraw[black] (axis cs: 0,0) circle (2pt) node[anchor = north east] {$p$};
        \filldraw[black] (axis cs: 4,3.125) circle (2pt) node[anchor = north west] {$q$};
    \end{axis}
    \end{tikzpicture}
    \caption{The curves $\gamma$ (dashed) and $\beta_2$ (solid).}
    \label{fig:wick2b}
    \end{subfigure}
    \caption{Illustration of the proof of Lemma \ref{lem:wick2}.}
\end{figure}


\subsection{Proof of Theorem~\ref{thm:nonsmooth}}\label{ssec:equiv}

In Section~\ref{ssec:wick} we have already shown that the bi-Lipschitz estimates~\eqref{eq:equivnonsmooth} of Theorem~\ref{thm:nonsmooth} hold locally if we choose $\tau$ to be a temporal function and $h=g_W$ to be the corresponding Wick-rotated Riemannian metric. We first extend these local results to weak temporal functions and arbitrary Riemannian metrics, and then prove Theorem~\ref{thm:nonsmooth} on compact sets.

\begin{lem}\label{lem:localbiLip}
 Let $(M,g)$ be a spacetime equipped with a weak temporal function $\tau$. Suppose $h$ is a Riemannian metric on $M$. Then, for every $x \in M$ there exists a neighborhood $U$ of $x$ and a constant $C\geq 1$ such that for all $p,q \in U$
 \begin{align}\label{eq:dhdtaulocal}
  \frac{1}{C} d_h(p,q) \leq \nulld_\tau(p,q) \leq C d_h(p,q).
 \end{align}
\end{lem}

\begin{proof}
 Let $x \in M$. We consider both inequalities in \eqref{eq:dhdtaulocal} separately for different Riemannian metrics and choose $U$ to be intersection of the neighborhoods $V_1$ and $V_2$ derived in each step. More precisely, the first inequality follows from the local anti-Lipschitz assumption for $\tau$ in Definition~\ref{LaLdef}, and the second inequality from Lemma~\ref{lem:wick2} for an auxiliary temporal functions. Without loss of generality we furthermore assume that $U$ is relatively compact. Then by a result of Burtscher~\cite{Bur}*{Thm.\ 4.5} the bi-Lipschitz estimate extends to any Riemannian metric.
 
 \medskip
 \textbf{Step 1. Lower bound}. We show that $\frac{1}{C} d_h(p,q) \leq \nulld_\tau(p,q)$ for $C$ and $h$ as in Definition~\ref{LaLdef}.
 
 Let $U_1$ be the neighborhood of $x$ from Definition~\ref{LaLdef}. Since $\nulld_\tau$ induces the manifold topology there is a radius $r>0$ such that for the open ball $\hat B_{3r}^\tau(x) \subseteq U_1$. Consider $p,q \in V_1 := \hat B_{r}^\tau(x)$. There exists a sequence $(\beta_n)_n$ of piecewise causal curves in $M$ between $p$ and $q$ such that
 \[
  \nullL_\tau(\beta_n) \leq \nulld_\tau(p,q) + \frac{1}{n}.
 \]
 For all $n> \frac{1}{r}$ the curves $\beta_n$ cannot leave $U_1$. Consider one such $\beta_n \colon [0,1]\to U_1$ and its partition $0 = s_0 < s_1 < \ldots < s_{k-1} < s_k =1$ in causal pieces. Then for $C$ and $h$ as in Definition~\ref{LaLdef}
 \begin{align*}
  \frac{1}{C} d_h(p,q)
  &\leq \frac{1}{C} \sum_{i=1}^k d_h (\beta_n(s_i),\beta_n(s_{i-1})) \\
  &\leq \sum_{i=1}^k \nulld_\tau (\beta_n(s_i),\beta_n(s_{i-1}))
  = \nullL_\tau(\beta_n) \leq \nulld_\tau(p,q) + \frac{1}{n}.
 \end{align*}
 Thus the result follows as $n\to\infty$.
 
 \medskip
 \textbf{Step 2. Upper bound.} We show that $\nulld_\tau(p,q) \leq 5 C d_W(p,q)$ for $g_W$ the Riemannian metric \eqref{gW} with respect to an auxiliary temporal function and corresponding $C$ of Definition~\ref{LaLdef}.
 
 Every spacetime which admits a time function also admits a temporal function $\tilde\tau$ by Bernal and S\'anchez~\cite{BeSa}*{Thm.\ 1.2}. We consider the Wick-rotated Riemannian metric $g_W$ defined in \eqref{gW} with respect to such a fixed $\tilde\tau$. Let $V_2$ be a geodesically convex neighborhood of $x$ with respect to  $g_W$. Without loss of generality we assume that $V_2$ is contained in the neighborhoods $U_2$ of Definition~\ref{LaLdef} and Lemma~\ref{lem:wick2}. Suppose $p,q \in V_2$ and $\gamma$ is the $g_W$-length minimizing geodesic $\gamma$ between $p$ and $q$ in $V_2$. In the proof of Lemma~\ref{lem:wick2} a piecewise causal curve $\beta$ in $V_2$, sufficiently close to $\gamma$, was constructed for which
 \begin{align}\label{eq:LW}
  L_W(\beta) &= \int_0^L \| \dot\beta(s) \|_{W} ds = \int_0^L \sqrt{|\dot\beta^{\tilde\tau}|^2 + \g(\dot\beta,\dot\beta)} ds \nonumber \\
  & \leq \int_0^L \sqrt{4^2 + 2} ds < 5 L = 5 L_W(\gamma).
 \end{align}
 Since $\beta$ is piecewise causal, there is a partition $0 = s_0 < s_1 < \ldots < s_{k-1} < s_k = L$ such that
 \[
  \nullL_\tau(\beta) = \sum_{i=1}^k \nulld_\tau(\beta(s_i),\beta(s_{i-1})),
 \]
 and $\tau$ being a weak temporal function implies that there is a constant $C\geq1$ (Definition~\ref{LaLdef}) such that
 \[
  \nullL_\tau(\beta) \leq C \sum_{i=1}^k d_W(\beta(s_i),\beta(s_{i-1})) \leq C L_W(\beta).
 \]
 Together with \eqref{eq:LW} and the fact that $\gamma$ is $g_W$-length minimizing we thus obtain
 \[
  \nulld_\tau(p,q) \leq \nullL_\tau(\beta) \leq C L_W(\beta) \leq 5C L_W(\gamma) = 5C d_W(p,q). \qedhere
 \]
\end{proof}

Lemma~\ref{lem:localbiLip} shows that the original local bi-Lipschitz bound \eqref{LaL} for causality related points (as required for weak temporal functions in Definition~\ref{LaLdef}) extends to the bi-Lipschitz bound \eqref{eq:dhdtaulocal} on a whole neighborhood of each point. Finally, we adapt the proof of Burtscher \cite{Bur}*{Thm.\ 4.5} to extend this local estimate to compact sets.

\begin{proof}[Proof of Theorem~\ref{thm:nonsmooth}] \label{proof:nonsmooth}
 Let $K$ be a compact subset of $M$. The argument proceeds by contradiction. Suppose that for any $n \in \nat$, there exist points $p_n , q_n \in K$ such that
 \begin{align}\label{ineq12}
  \nulld_{\tau}(p_n,q_n) > n d_h(p_n, q_n).
 \end{align}
 By passing to subsequences, we can assume that $p_n \to p$ and $q_n \to q$ (in one and hence both metrics, since they both induce the manifold topology). Since $K$ is compact it is bounded with respect to the null distance $\nulld_{\tau}$ and the inequality~\eqref{ineq12} furthermore implies that
 \[
  d_h(p_n, q_n) \to 0 \quad\text{as} \quad n \to \infty,  
 \]
 and hence $p=q$. Thus, for $n$ large enough, $p_n, q_n$ are contained in neighborhood $U$ of $p$ that by Lemma~\ref{lem:localbiLip} is small enough so that \eqref{eq:dhdtaulocal} holds for some $\tilde{C} \geq 1$. But then we conclude that
 \[
  \tilde{C} d_h(p_n,q_n) \geq \nulld_{\tau}(p_n,q_n) > n d_h(p_n,q_n),
 \]
 a contradiction for $n$ large. This proves that there is a constant $C \geq 1$ such that $\nulld_\tau(p,q) \leq C d_h(p,q)$, and the reverse inequality is obtained in the same way.
\end{proof}

\begin{rem}
 Bi-Lipschitz maps are a crucial tool in Metric Geometry. In particular, they are useful in the study of Gromov--Hausdorff convergence of metric spaces. Moreover, the estimates of Theorem~\ref{thm:nonsmooth} can be used to equip suitable spacetimes with a (local) integral current space structure (see \cites{SoWe,AmKi,JaLe,La,LaWe,FeFl}) via the approach laid out in Allen and Burtscher~\cite{AlBu}*{Sec.\ 2.5, 2.6, and 4}. Thanks to our Completeness Theorem~\ref{mainthm} and \cite{AlBu}*{Thm.\ 1.3} we know, for instance, that any globally hyperbolic spacetime viewed as the metric space $(M,\nulld_\tau)$ with completely uniform temporal function $\tau$ is a (local) integral current space (see Section~\ref{sec:complete}). This also makes it possible to use spacetime intrinsic flat convergence to study geometric stability questions in General Relativity, as proposed by Sormani.
\end{rem}


\subsection{Direct proof of Corollary~\ref{equivthm1}}\label{ssec:altequivtauproof}

Corollary~\ref{equivthm1} is a very special case of Theorem~\ref{thm:nonsmooth}. Here we show that it can also be obtained directly without the use of temporal functions. We first prove a local result.

\begin{lem} \label{equivlemma}
 Let $\tau_1$, $\tau_2$ be two weak temporal functions (see Definition~\ref{LaLdef}) on a spacetime $(M,g)$, and $\hat d_{\tau_1}$, $\hat d_{\tau_2}$ their associated null distances. Then, for every point $x \in M$, there exists a neighborhood $U$ of $x$ and a constant $C \geq 1$ such that for all $p,q \in U$,
 \[
  \frac{1}{C} \hat d_{\tau_1}(p,q) \leq \hat d_{\tau_2}(p,q) \leq C \hat d_{\tau_1}(p,q).
 \]
\end{lem}

\begin{proof}
 Fix a point $x \in M$ and an arbitrary Riemannian metric $h$ on $M$. Let $U_x^{\tau_1}$ and $U_x^{\tau_2}$ be the neighborhoods around $x$ where both $\tau_1$ and $\tau_2$ are Lipschitz and anti-Lipschitz with positive constants $C_1,C_2$ with respect to $h$ as in \eqref{LaL} of Definition~\ref{LaLdef}. Since $M$ is locally compact there exists a relatively compact neighborhood $V \subseteq U_x^{\tau_1}\cap U_x^{\tau_2}$ of $x$. Since both null distances $\nulld_{\tau_1},\nulld_{\tau_2}$ induce the manifold topology \cite{SoVe}*{Thm.\ 4.6}, we find $r>0$ sufficiently small so that $\hat B_{4r}^{\tau_1}(x) \cup \hat B_{4r}^{\tau_2}(x) \subseteq V$. Define a neighborhood of $x$ by \[U := \hat B_{r}^{\tau_1}(x) \cap \hat B_{r}^{\tau_2}(x).\]

We claim that for $i=1,2$ and for all $p,q\in U$,
 \begin{align}\label{dtau1}
  \hat d_{\tau_i}(p,q) = \inf \{ \nullL_{\tau_i}(\beta) \mid \beta \text{ a p.w.\ causal curve in } V\text{ between } p \text{ and } q\}.
 \end{align}
 What we are saying is that the null distances in $U$ can be approximated by sequences of curves that never leave $V$. To prove this, let $p,q \in U$, consider an arbitrary $\varepsilon \in (0,r)$, and for each $i = 1,2$ choose a piecewise causal curve $\beta_\varepsilon^i \colon I \to M$ from $p$ to $q$ such that
 \[
  \nullL_{\tau_i}(\beta_\varepsilon^i) < \nulld_{\tau_i}(p,q) + \varepsilon.
 \]
 Then, for all $t \in I$,
 \begin{align*}
  \nulld_{\tau_i}(x,\beta^i_\varepsilon(t)) &\leq \nulld_{\tau_i}(x,p) + \nulld_{\tau_i}(p, \beta^i_\varepsilon (t)) \leq \nulld_{\tau_i}(x,p) + \nullL_{\tau_i}(\beta^i_\varepsilon)\\ &\leq \nulld_{\tau_i}(x,p) + \nulld_{\tau_i}(p,q) + \varepsilon < 4r,
 \end{align*}
 hence $\beta^i_\varepsilon$ lies entirely in $\hat B^{\tau_i}_{4r}(x) \subseteq V$, proving the claim \eqref{dtau1} for $i=1,2$.
 
 For any causally related points $(\tilde p,\tilde q) \in \J$ in $ U_x^{\tau_1} \cap U_x^{\tau_2}$, by \eqref{LaL}, we have for $C := C_1 C_2\geq 1$ that
 \begin{equation} \label{equivkeyeq}
  \frac{1}{C} (\tau_1(\tilde q) - \tau_1(\tilde p)) \leq \tau_2(\tilde q) - \tau_2(\tilde p) \leq {C} (\tau_1(\tilde q) - \tau_1(\tilde p)).
 \end{equation}
 This implies that for all piecewise causal curves $\beta$ contained in $U_x^{\tau_1} \cap U_x^{\tau_2}$,
 \begin{align}\label{eqC}
  \frac{1}{C} \nullL_{\tau_1} (\beta) \leq \nullL_{\tau_2}(\beta) \leq {C} \nullL_{\tau_1}(\beta).
 \end{align}
 It remains to be shown that the bi-Lipschitz bounds on the null lengths extends to the null distances between any two points $p$ and $q$ in the smaller set $U$. We have already seen that for each $n \in \nat$ exists a piecewise causal curve $\beta^1_n$ between $p$ and $q$ (by \eqref{dtau1} $\beta^1_n$ is entirely contained in $V$!) such that
 \[
  \nullL_{\tau_1}(\beta^1_n) \leq \nulld_{\tau_1}(p,q) + \frac{1}{n},
 \]
 and by \eqref{eqC} therefore
 \[
  \nulld_{\tau_2} (p,q) \leq \nullL_{\tau_2}(\beta^1_n) \leq C \nullL_{\tau_1}(\beta^1_n) \leq C \nulld_{\tau_1} (p,q) + \frac{C}{n}.
 \]
 Since this inequality holds for all $n$ and fixed $C$, we have that $\nulld_{\tau_2}(p,q) \leq C \nulld_{\tau_1}(p,q)$ on $U$.
 In the same way, for piecewise causal curves $\beta^2_n$ approximating $\nulld_{\tau_2}(p,q)$, we obtain $\nulld_{\tau_1}(p,q) \leq C \nulld_{\tau_2}(p,q)$ on $U$, which proves the Lemma.
\end{proof}

\begin{proof}[Proof of Corollary~\ref{equivthm1}]
 Having established the local result in Lemma~\ref{equivlemma}, the proof on compact sets is now the same as that of Theorem~\ref{thm:nonsmooth}.
\end{proof}


\subsection{Counterexamples}
\label{ssec:equivcount}

Finally, we construct two examples showing that the assumptions and statement of Corollary~\ref{equivthm1}, and hence also that of Theorem~\ref{thm:nonsmooth} and Corollary~\ref{cor:differentg}, are sharp.

\begin{ex}[Corollary~\ref{equivthm1} is false globally]\label{ex:globalnotequiv}
 Consider the Minkowski spacetime $M=\real^{1,n}$ with the temporal functions $\tau_1 = t$ and $\tau_2 = \exp(t)$ (which even have the same level sets!). Since the exponential function is \textit{not} globally Lipschitz, the time functions $\tau_1$ and $\tau_2$ are not equivalent in the sense that the bi-Lipschitz condition \eqref{equivkeyeq} for causally related points does not hold globally. Since $\tau(q) - \tau(p) = \nulld_\tau(p,q)$ for all $q \in J^+(p)$, it follows that if \eqref{equivkeyeq} fails to hold globally, then the null distances $\nulld_{\tau_1}$, $\nulld_{\tau_2}$ are globally inequivalent.
\end{ex}

Example~\ref{ex:globalnotequiv} shows that although $\tau_1$ and $\tau_2$ are both temporal functions on the \emph{same} spacetime $(M,g)$, there is a clear distinction \emph{globally} between the metric structures induced by the null distances $\nulld_{\tau_1}$ and $\nulld_{\tau_2}$. In particular, in Section \ref{sec:complete} we show that $(M,\nulld_{\tau_1})$ is complete while $(M,\nulld_{\tau_2})$ is not.

\begin{ex}[Corollary~\ref{equivthm1} is false for time functions that are not locally Lipschitz] \label{ex:nonlocLip}
 Consider the Minkowski spacetime $M = \real^{1,n}$ with respect to the usual global coordinates $(t,x)$. Let $\tau_1(t,x):=t$ be the standard (smooth and locally anti-Lipschitz) time function and \[ \tau_2(t,x):=\operatorname{sgn}(t) \sqrt{|t|}. \] Clearly, $\tau_2$ is continuous and a time function. It is, however, \emph{not} locally Lipschitz at $t=0$ because $\partial_t \tau_2(t,x) \to \infty$ as $t \to 0$. Nonetheless, $\tau_2$ is locally anti-Lipschitz (and therefore $\hat d_{\tau_2}$ induces the manifold topology): Consider a point $(t,x)$ with $t>0$ and the neighborhood $U = (0,2t) \times \real^n$. Suppose $q \in J^+(p)$ with $p=(t_p,x_p)$, $q=(t_q,x_q)$. By the Mean Value Theorem for $\sqrt{\phantom{.}}$ on $[0,2t]$ we have
 \begin{align*}
  \sqrt{t_q}-\sqrt{t_p} \geq \left( \inf_{r\in(0,2t)} \frac{d}{dr}\sqrt{r} \right) (t_q-t_p) = \frac{1}{2\sqrt{2t}} |t_q-t_p|.
 \end{align*}
 Hence due to the causal relation on Minkowski spacetime we obtain with respect to the Euclidean distance $d$ and with $C := \frac{1}{4\sqrt{2t}}>0$ that on $U$
 \[
  q \in J^+(p) \Longrightarrow \tau_2(q) - \tau_2(p) \geq \frac{1}{4\sqrt{2t}} \sqrt{(t_q-t_p)^2 + |x_q-x_p|^2} \geq  C d(q,p).
 \]
 In the same fashion, $\tau_2$ is anti-Lipschitz in the neighborhood $U = (2t,0) \times\real$ of a point $(t,x)$ with $t<0$. If $t=0$, we can simply use the neighborhood $U = (-1,1) \times \real^n$ and $C=\frac{1}{4}$ (since the infimum is at $|r| =1$). Hence $\tau_2$ is locally anti-Lipschitz everywhere (but clearly not globally anti-Lipschitz with respect to the Euclidean distance).

 Now assume that $p=(0,0)$ and let $q_t=(t,0)$, $t \in (0,1)$, be arbitrarily close to $p$. Then
 \[
  \hat d_{\tau_1}(p,q_t) = t \leq \sqrt{t} = \hat d_{\tau_2} (p,q_t)
 \]
 but the bi-Lipschitz estimate \eqref{eq:equivnonsmooth} of Theorem~\ref{equivthm1} does not hold because
 \[
  \frac{\hat d_{\tau_2}(p,q_t)}{\hat d_{\tau_1}(p,q_t)} = \frac{1}{\sqrt{t}} \to \infty \quad \text{ as } t \to 0.
 \]
\end{ex}


\subsection{Basic properties of weak temporal functions}\label{ssec:gradient}

The following result justifies the notion of weak temporal functions.

\begin{prop} \label{prop:timelikegrad}
 Let $(M,g)$ be a spacetime and $\tau \colon M\to\real$ be a weak temporal function. Then $\tau$ is locally Lipschitz with past-directed timelike gradient $\nabla\tau$ almost everywhere.
\end{prop}

\begin{proof}
 The local Lipschitz condition of Definition~\ref{LaLdef} implies the upper bound in Lemma~\ref{lem:localbiLip}, i.e., for each Riemannian metric $h$ and each point of $M$ there is an open neighborhood $U$ and $C\geq 1$ such that for $p,q\in U$
 \[
  |\tau(q)-\tau(p)| \leq \nulld_\tau(p,q) \leq C d_h(p,q).
 \]
 In other words, $\tau$ is locally Lipschitz on $M$. By Rademacher's Theorem $\nabla\tau$ therefore exits almost everywhere.
 
 Suppose that $\nabla\tau$ exists at point $p$. Consider a future-directed causal vector $v \in T_pM \setminus \{0\}$ and a smooth causal curve $\gamma$ from $p =\gamma(0)$ in direction $v$. Then
 \begin{align*}
  d\tau(v) = \left. \frac{d}{ds} \right|_{s=0} (\tau\circ\gamma)(s) = \lim_{s\to 0} \frac{\tau(\gamma(s))-\tau(p)}{s} = \lim_{s\to 0} \frac{\nulld_\tau(\gamma(0),\gamma(s))}{s},
 \end{align*}
 which by the local anti-Lipschitz property of $\tau$ and \cite{Bur}*{Prop.\ 4.10} implies that
 \[
  d\tau(v) \geq \frac{1}{C} \lim_{s\to 0} \frac{d_h(\gamma(0),\gamma(s))}{s} = \frac{1}{C} \| v\|_h >0.
 \]
 Hence $\nabla\tau$ is past-directed timelike whenever it exists.
\end{proof}

The converse is not true, in the sense that a function with almost everywhere timelike gradient is not necessarily weak temporal, unless one assumes local upper and lower bounds on $\Vert \nabla \tau \Vert_h$. Indeed, the time functions in Example~\ref{ex:nonlocLip} and \cite{SoVe}*{Ex.\ 3.4} have timelike gradient almost everywhere, but the first is not locally Lipschitz and the second not locally anti-Lipschitz.

\medskip
The following important classes of time functions are weak temporal.

\begin{lem}\label{temporalareLaL}
 Temporal functions and regular cosmological time functions (\`{a} la Andersson--Galloway--Howard~\cite{AGH} and Wald--Yip~\cite{WaYi}) are weak temporal functions.
\end{lem}
 
\begin{proof}
 For temporal functions, locally Lipschitz follows by smoothness and locally anti-Lipschitz by both \cite{SoVe}*{Cor.\ 4.16} and \cite{CGM}*{Prop.\ 4.3} (together with Galloway's observation of compatibility \cite{SoVe}*{p.\ 19}). For regular cosmological time functions, locally Lipschitz follows by \cite{AGH}*{Thm.\ 1.2(v)} and locally anti-Lipschitz by \cite{SoVe}*{Thm.\ 5.4}.
\end{proof}


\section{Encoding causality}\label{sec:causality}

In this section we prove Theorem~\ref{thm:admcausalityglobal} and a corresponding local result.
 By Definition \ref{def:nulld} of the null distance for any $p,q \in M$
\[
 \nulld_\tau(p,q) \geq |\tau(q)-\tau(p)|.
\]
Equality holds trivially for causally related points, i.e.,
\[(p,q) \in J^+ \Longrightarrow \nulld_\tau(p,q) = \tau(q) -\tau(p).\]
where $\J \subseteq M \times M$ denotes the causal relation
\[\J := \{(p,q) \mid \text{there exists a future-directed causal curve from $p$ to $q$}\}.\] 
We investigate when the converse holds, that is, when
\begin{equation} \label{eq:causenc}
  (p,q) \in \J \Longleftrightarrow \nulld_\tau(p,q) = \tau(q) - \tau(p),
\end{equation}
in which case the null distance is said to \emph{encode causality}, an open problem mentioned in \cite{SoVe}*{Sec.\ 1}. We introduce a new notation for the right hand side of \eqref{eq:causenc}.

\begin{defn}
 Let $(M,g)$ be a spacetime equipped with a time function $\tau \colon M \to \real$. We define the \emph{null distance relation} $\nullR_\tau \subseteq M \times M$ by
 \[
  (p,q) \in \nullR_\tau :\iff \nulld_\tau(p,q) = \tau(q) - \tau(p).
 \]
\end{defn}

Clearly, the relation $\nullR_\tau$ is reflexive and transitive. Antisymmetry requires definiteness of $\nulld_\tau$ which, for instance, follows if $\tau$ is a locally anti-Lipschitz time function~\cite{SoVe}*{Thm.\ 4.6}. Continuity of $\tau$ and $\nulld_\tau$ imply closedness of $\nullR_\tau$. We can thus summarize the basic properties of $\nullR_\tau$ as follows.

\begin{lem}\label{lem:Rprop}
 Suppose $(M,g)$ is a spacetime and $\tau$ is a locally anti-Lipschitz time function on $M$. Then the null distance relation $\nullR_\tau$ is a closed partial order on $M$ satisfying $\J \subseteq \nullR_\tau$. \qed
\end{lem}

In Section~\ref{ssec:causalityglobal} we prove that globally hyperbolic spacetimes encode causality when $\tau$ is carefully chosen, which can be reformulated in terms of $\nullR_\tau$ as follows.

\begin{thm}\label{thm:causalityglobal}
 Let $(M,g)$ be a globally hyperbolic spacetime and let $\tau$ be a temporal function such that every nonempty level set is a Cauchy surface. Then causality is encoded in $\nulld_\tau$, meaning $\J = \nullR_\tau$.
\end{thm}

Note that the time functions $\tau \colon M \to (T_1,T_2)$ of Theorems~\ref{thm:admcausalityglobal} and \ref{thm:causalityglobal} are Cauchy time functions (as per Theorem~\ref{thm:Cauchytime}) up to composition with an increasing homeomorphism $f \colon (T_1,T_2) \to \real$. Our relaxed Cauchy assumption adds value in the weak context, notably because cosmological time functions take values only on $(0,\infty)$.

Subsequently we also prove a local version of this result for any spacetime with a temporal function. 

\begin{thm}\label{thm:causalitylocal}
 Let $(M,g)$ be a spacetime and $\tau$ be a temporal function. Then, every point $x\in M$ has a neighborhood $U$ such that 
 \begin{align*}
     J^+ \cap (U \times U) = \nullR_\tau \cap (U \times U). 
 \end{align*}
\end{thm}

Upon completion of an earlier preprint of this manuscript and that of Sakovich and Sormani~\cite{SaSo} we noticed that our different proofs can be combined to yield a stronger result of both Theorem~\ref{thm:causalityglobal} (to locally anti-Lipschitz $\tau$) and \cite{SaSo} (to noncompact Cauchy slices). This stronger result is Theorem~\ref{thm:admcausalityglobal} and proven in Section~\ref{ssec:admcausalityglobal}.

In Section~\ref{ssec:causalitycount} we show that the assumptions for the global Theorem~\ref{thm:admcausalityglobal} are sharp in the sense that there are globally hyperbolic spacetimes and locally anti-Lipschitz functions $\tau$ for which $\J \subsetneq \nullR_\tau$.


\subsection{Proofs of Theorems~\ref{thm:causalityglobal} and \ref{thm:causalitylocal} for temporal functions}
\label{ssec:causalityglobal}
 
In this section, the Wick-rotated metric $g_W$ of $g$ introduced in Section \ref{ssec:wick} plays an important role again. Since $g_W$ need not be complete even if the spacetime is globally hyperbolic (see Example \ref{ex:notcomplete}) we furthermore make use of the existence of a conformally equivalent complete Riemannian metric
 \begin{align}\label{gRgW}
 g_R := \Omega^2 g_W,
 \end{align}
 where $\Omega \colon M \to [1,\infty)$ is a smooth function \cite{NoOz}*{Thm.~1}. We denote the corresponding norm, length functional and distance by $\left\Vert\cdot\right\Vert_R$, $L_R$ and $d_R$, respectively. The relation of $g_R$ and $g_W$ carries over to the distances as follows, that is, for any curve $\gamma \colon [a,b] \to M$,
 \[
  L_W(\gamma) \leq L_R(\gamma),
 \]
 and hence for any $p,q \in M$
 \[
  d_W(p,q) \leq d_R(p,q).
 \]
In the following lemma, we obtain reverse inequalities on compact sets.
 
\begin{lem}\label{lem:complete}
  Let $g_R$ and $g_W$ be two conformally equivalent Riemannian metrics on $M$ as in \eqref{gRgW}, and suppose $K$ is a compact set in $M$. Then there exist constants $C_K \geq c_K \geq 1$ such that for any curve $\gamma$ in $K$
 \begin{align*}
  L_W(\gamma) \leq L_R(\gamma) \leq c_K L_W(\gamma).
 \end{align*}
 and for any $p,q \in K$,
 \[
  d_W(p,q) \leq d_R(p,q) \leq C_K d_W(p,q).
 \]
\end{lem}

\begin{proof}
  The first statement about the lengths follows immediately from the Definition \eqref{gRgW}, and we may pick $c_K := \max_{x \in K}\Omega(x)$. The first inequality for the distances is trivial and the second inequality follows from \cite{Bur}*{Thm.~4.5} (note that, in general, we need to pick $C_K > c_K$ as minimizing curves may leave $K$, but the proof in \cite{Bur} guarantees boundedness of $C_K$).
\end{proof}

With these tools, we proceed to prove the theorem.
 
\begin{proof}[Proof of Theorem~\ref{thm:causalityglobal}]
 Trivially, $\J \subseteq \nullR_\tau$, so we only need to prove $\nullR_\tau \subseteq \J$. Suppose that $(p,q) \in \nullR_\tau$, that is, $\nulld_\tau(p,q) = \tau(q)-\tau(p)$. Since $(p,p) \in \J$ is trivially satisfied, we can assume that $p \neq q$ and thus necessarily $\tau(q) > \tau(p)$. By definition of $\nulld_\tau$ there exists a sequence of piecewise causal curves $(\beta_n)_n$ such that
 \begin{align}\label{eq:globminimizing}
  0 < \tau(q) - \tau(p) \leq \nullL_\tau(\beta_n) \leq \tau(q) - \tau(p) + \frac{1}{n}.
 \end{align}

 In what follows we construct a future-directed causal curve $\beta$ between $p$ and $q$, which then immediately implies $(p,q) \in \J$. We proceed as follows: After a preliminary local estimate we construct a candidate limit curve near $p$ in $(M,g_R)$, and then show that it is both locally Lipschitz and future-directed causal. Finally, we show it naturally extends all the way up to $q$.
 
 \medskip
 \textbf{Step 1. A local version of \eqref{eq:globminimizing}.}
We consider $\beta_n \colon [0,L_n] \to M$ on arbitrary subintervals of its domain.
Suppose there exists a subinterval $[s,t] \subseteq [0,L_n]$ such that $\nullL_\tau(\beta_n|_{[s,t]}) > \tau(\beta_n(t))-\tau(\beta_n(s)) + \frac{1}{n}$. Then by the additivity of null lengths
\begin{align*}
 \nullL_\tau(\beta_n) &= \nullL_\tau(\beta_n|_{[0,s]}) + \nullL_\tau(\beta_n|_{[s,t]}) + \nullL_\tau(\beta_n|_{[t,L_n]}) \\
 &> |\tau(\beta_n(s))-\tau(p)| + | \tau(\beta_n(t))-\tau(\beta_n(s))  + \frac{1}{n} | + |\tau(q)-\tau(\beta_n(t))| \\
 &\geq \tau(q)-\tau(p) +\frac{1}{n},
\end{align*}
a contradiction to \eqref{eq:globminimizing}. Thus for all subintervals $[s,t]$ of the domain of $\beta_n$
\begin{align}\label{eq:locminimizing}
|\tau(\beta_n(t)) -\tau(\beta_n(s))| \leq \nullL_\tau(\beta_n|_{[s,t]}) \leq \tau(\beta_n(t)) - \tau(\beta_n(s)) + \frac{1}{n}.
\end{align}
 (Note that at this point we cannot yet rule out that $\tau(\beta_n(t)) - \tau(\beta_n(s))< 0$, which is why we do not want to drop the absolute value on the left hand side.)

 \medskip
 \textbf{Step 2. Construction of a candidate limit curve $\beta$ near $p$.}
 Suppose, that each $\beta_n \colon [0,L_n] \to M$ is parametrized by $g_R$-arclength with $\beta_n(0)=p$ and $\beta_n(L_n)=q$. In addition, we attach a future-directed inextendible causal curve $\tilde\beta$ at $q$. Since $g_R$ is complete, $\tilde\beta$ must have infinite $g_R$-length.
 Thus we can extend all $\beta_n$ by $\tilde\beta$ to obtain future-inextendible piecewise causal curves with $g_R$-arclength parametrization, again denoted by $\beta_n \colon [0,\infty) \to M$ and satisfying \eqref{eq:locminimizing}. In particular, for each $n$ and any $s,t \in [0,\infty)$,
 \begin{align}\label{eq:Lip}
  d_R(\beta_n(s),\beta_n(t)) \leq L_R(\beta_n|_{[s,t]}) = |t-s|.
 \end{align}
 
 By construction, $\beta_n (0)=p$ for all $n$. Fix any other $t_0 \in (0,\infty)$. Then, due to the $g_R$-arclength parametrization, each curve segment $\beta_n|_{[0,t_0]}$ is contained in the bounded set (which, by the Hopf--Rinow Theorem for $g_R$ is also compact)
 \[
  B_{t_0}(p) := \{ x \in M : d_R(p,x) \leq t_0 \}.
 \]
 In particular, the family $(\beta_n|_{[0,t_0]})_n$ is uniformly bounded and uniformly equicontinuous. Thus, by the Arzel\`{a}--Ascoli Theorem (see, for instance, \cite{Mun}*{Thm.\ 47.1} or \cite{BEE}*{Thm.\ 3.30}), there exists a continuous curve $\beta \colon [0,\infty) \to M$, such that a subsequence of $(\beta_{n})_n$ converges uniformly to $\beta$ on all compact subintervals. Since $\frac{1}{n_k}\leq \frac{1}{k}$ for any subsequence, we can denote this subsequence again by $(\beta_n)_n$. Moreover, \eqref{eq:Lip} implies that $\beta$ is a locally Lipschitz curve with
 \begin{align}\label{eq:Lip1}
  d_R(\beta(s),\beta(t)) \leq |t-s|
 \end{align}
 for all $s,t \in [0,\infty)$.

\medskip
\textbf{Step 3. The curve $\beta$ is future-directed causal.} Since $\beta$ is locally Lipschitz, together with Rademacher's Theorem, we know that $\dot\beta$ exists almost everywhere. To conclude that $\beta$ is a future-directed causal curve, it remains to be shown that $g(\dot \beta,\dot \beta) \leq 0$ and $(\tau \circ \beta)'>0$ almost everywhere.

By \cite{Bur}*{Prop.\ 4.10} the $g_R$-norm of the analytic derivative and $d_R$-metric derivative of $\beta$ exist and coincide almost everywhere. Combined with the Lipschitz estimate~\eqref{eq:Lip1} for $\beta$ we thus have for almost all $s \in [0,\infty)$
\begin{align}\label{gwbound}
 0 \leq \| \dot\beta(s) \|_R = \lim_{h \to 0} \frac{d_R(\beta(s+h),\beta(s))}{|h|} \leq \lim_{h\to 0} \frac{|h|}{|h|} = 1.
\end{align}
In order to show that $\dot\beta$ is almost everywhere causal we need to control the $\tau$-component $\dot\beta^\tau = (\tau \circ \beta)'$ of the tangent vectors. Suppose $\dot\beta(s_0)$ exists for a fixed $s_0 \in (0,\infty)$, and consider the closed (and hence compact) $d_R$-ball $B_\varepsilon$ of radius $\varepsilon$ at $\beta(s_0)$ in $M$. By Lipschitz continuity \eqref{eq:Lip1} of $\beta$ the whole interval $[s_0-\varepsilon/2,s_0+\varepsilon/2]$, is mapped into $B_{\varepsilon/2}$. We use the approximating sequence $(\beta_n)_n$ of $\beta$ next, more precisely, that the $\beta_n$ are piecewise causal and converge uniformly on compact intervals. Thus for sufficiently large $n$ all $\beta_n([s_0-\varepsilon/2,s_0+\varepsilon/2]) \subseteq B_\varepsilon$. By the local estimate \eqref{eq:locminimizing}
\begin{align}\label{Ltn}
 \lim_{n\to\infty} \nullL_\tau(\beta_n|_{[s_0-\varepsilon/2,s_0+\varepsilon/2]}) = \tau(\beta(s_0+\varepsilon/2)) - \tau(\beta(s_0-\varepsilon/2)).
\end{align}
On the other hand, due to the $g_R$-arclength parametrization of $\beta_n$ and Lemmas~\ref{lem:wick} and \ref{lem:complete} (with constant $c_\varepsilon = \max_{x \in B_\varepsilon} \Omega(x)\geq 1$)
\begin{align*}
 \varepsilon = L_R(\beta_n|_{[s_0-\varepsilon/2,s_0+\varepsilon/2]}) &\leq c_\varepsilon L_W(\beta_n|_{[s_0-\varepsilon/2,s_0+\varepsilon/2]}) \\ &\leq \sqrt{2} c_\varepsilon \nullL_\tau(\beta_n|_{[s_0-\varepsilon/2,s_0+\varepsilon/2]}).
\end{align*}
Hence by \eqref{Ltn}
\[
 \frac{1}{\sqrt{2} c_\varepsilon} \leq \frac{1}{\varepsilon} \lim_{n\to\infty} \nullL_\tau(\beta_n|_{[s_0-\varepsilon/2,s_0+\varepsilon/2]}) = \frac{\tau(\beta(s_0+\varepsilon/2)) - \tau(\beta(s_0-\varepsilon/2))}{\varepsilon}.
\]
Due to the continuity of $\Omega$, it follows that $c_\varepsilon \to \Omega(\beta(s_0))$ as $\varepsilon \to 0$, while the difference quotient of $\tau \circ \beta$ converges to the derivative $\dot\beta^\tau(s_0) = (\tau \circ \beta)'(s_0)$. Thus in the limit we obtain \[\dot\beta^\tau(s_0) \geq \frac{1}{\sqrt{2} \Omega(\beta(s_0))} > 0.\]
Due to the Wick-rotation $g_W$ of $g$ as well as the $g_R$-bound \eqref{gwbound}
\begin{align*}
 g(\dot \beta(s_0),\dot \beta(s_0)) &= - 2|\dot\beta^\tau(s_0)|^2 + \|\dot\beta(s_0)\|^2_W \\ &\leq - \Omega^{-2}(\beta(s_0)) + \Omega^{-2}(\beta(s_0)) \| \dot\beta(s_0) \|^2_R \leq 0.
\end{align*}
This proves that $\dot\beta$ is future-directed causal almost everywhere.

\medskip
\textbf{Step 4. The point $q$ lies on $\beta$.}
By construction $\beta(0) = p$, and by Step 3 we know that $\tau\circ\beta$ is strictly increasing on $[0,\infty)$. We distinguish two cases:
\begin{enumerate}
 \item If there is an $s_0 \in [0,\infty)$ such that $\tau(\beta(s_0)) = \tau(q)$, then the following argument implies that $\beta(s_0) = q$:
 
Let $\varepsilon >0$ be arbitrary. Since both $\nulld_\tau$ and $d_R$ induce the manifold topology, there exists a $\delta>0$ such that $d_R(\beta(s_0),x) < \delta$ implies $\nulld_\tau(\beta(s_0),x) < \varepsilon$. Due to the convergence $\beta_{n}(s_0) \to \beta(s_0)$ (obtained with respect to $d_R$ in Step 2), for any $n$ sufficiently large
 \begin{align}\label{eq:s0est1}
  \nulld_\tau(\beta(s_0),\beta_n(s_0)) < \varepsilon.
 \end{align}
Moreover, due to the continuity of $\tau \circ \beta$, for all $n > \frac{1}{\varepsilon}$ sufficiently large
 \[| \tau(\beta_{n}(s_0)) - \tau(q)| = | \tau(\beta_{n}(s_0)) - \tau(\beta(s_0))| < \varepsilon.\]
Since $\beta_{n}(L_{n}) = q$, the local estimate \eqref{eq:locminimizing} on $[s_0,L_{n}]$ (or $[L_n,s_0]$ if $L_n < s_0$) yields
 \begin{align}\label{eq:soest2}
   \nulld_\tau(\beta_{n}(s_0),q) &\leq \nullL_\tau(\beta_{n}|_{[s_0,L_n]}) \nonumber \\
    &\leq \tau(q) - \tau(\beta_n(s_0)) + \frac{1}{n} < 2 \varepsilon.
 \end{align}
Combining \eqref{eq:s0est1} and \eqref{eq:soest2} implies that $\nulld_\tau(\beta(s_0),q) < 3 \varepsilon$ for any $\varepsilon >0$. Thus $\beta(s_0) = q$.

\item The only obstruction to the desired conclusion is therefore that for all $s \in [0,\infty)$
\begin{align}\label{sq}
 \tau(\beta(s)) < \tau(q).                  
 \end{align}
We show that this case cannot occur. Since the level set $\tau^{-1}(q)$ is a Cauchy surface and $\tau$ increases along $\beta$, \eqref{sq} implies that the causal curve $\beta$ is future extendible as a future-directed causal curve. In particular, the future endpoint $x := \lim_{s \to \infty} \beta(s)$ of $\beta$ exists (since it is necessarily part of any extension) and by \eqref{sq}, $\tau(x) \leq \tau(q)$. If $x=q$ we are done, otherwise there exists a relatively compact open set $W$ around $\beta \cup \{x\}$ such that
$q \not\in \overline{W}$.

We will use the approximating curves $\beta_n$ to show that $\beta$ must in fact leave $W$: Since $\beta_n(0)=p \in W$ and $\beta_n(L_n) = q \not\in W$ there exists
\[
 b_n := \sup\{ s \mid \beta_n(t) \in W \text{ for all } t \in [0,s] \} \in (0,L_n),
\]
and $\beta_n(b_n) \in \partial W$. Since $\overline{W}$ is compact, by Lemmas~\ref{lem:wick} and \ref{lem:complete}, there exists a constant $C>0$ such that
\begin{align*}
 b_n = L_R(\beta_n|_{[0,b_n]}) &\leq \sqrt{2} C \nullL_\tau(\beta_n|_{[0,b_n]}) \leq \sqrt{2} C \nullL_\tau(\beta_n|_{[0,L_n]})\\
 &\leq \sqrt{2} C \left( \tau(q) - \tau(p) + \frac{1}{n} \right).
\end{align*}
Hence all $b_n$ are uniformly bounded from above by a constant $a := \sqrt{2} C \left( \tau(q) - \tau(p) + 1 \right)$. In particular, a subsequence converges to $b := \limsup b_n \in [ 0, a)$. Let $\varepsilon >0$. By uniform convergence $\beta_n \to \beta$ on $[0,a]$ for all sufficiently large $n$ (along the previous subsequence), we have
\begin{align}\label{eq:gb1}
 d_R(\beta(b_n),\beta_n(b_n)) < \varepsilon.
\end{align}
Since $\beta_n(b_n) \in \partial W$ and $\partial W$ is compact as closed subset of the compact set $\overline{W}$, there exists a subsequence of points $\beta_{n_k}(b_{n_k})$ that converges to a point $y \in \partial W$, i.e., for $k$ sufficiently large
\begin{align}\label{eq:gb2}
 d_R(\beta_{n_k}(b_{n_k}),y) < \varepsilon.
\end{align}
Combining \eqref{eq:gb1}--\eqref{eq:gb2} yields
\begin{align*}
 d_R(\beta(b),y) &\leq d_R(\beta(b),\beta_{n_k}(b_{n_k})) + d_R(\beta_{n_k}(b_{n_k}), y) < 2 \varepsilon.
\end{align*}
In other words, $\beta(b) = y \in \partial W$, a contradiction to the assumption that $\beta$ is entirely contained in the open set $W$. Hence the assumption \eqref{sq} must be false, and thus by case (i) $\beta$ indeed reaches $q$.
\end{enumerate}
To sum up, we have constructed a future-directed causal (Step 3) curve $\beta$ from $p$ (Step 2) to $q$ (Step 4). Therefore, $(p,q) \in \J$, and thus $\J = \nullR_\tau$.
\end{proof}

\begin{rem}[Future/Past Cauchy level sets]\label{rem:fcauchy}
 Inspecting the proof of Theorem~\ref{thm:causalityglobal} one observes that the assumption that the level sets of $\tau$ are Cauchy is only used in Step 4(ii). In fact, it is only needed that the level sets are \emph{past Cauchy} because we construct the limiting curve $\beta$ from $p$ to $q$ (and $\tau(p)<\tau(q)$). We could have equally well constructed the curve from $q$ in which case we would have used that the level sets are \emph{future Cauchy} (see definition in \cite{AGH}*{p.~315}). Since either case yields the desired result, $\tau$ having future (or past) Cauchy level sets is already sufficient for $\nulld_\tau$ to encode causality.
\end{rem}

We can adapt the proof of Theorem~\ref{thm:causalityglobal} to show its local counterpart Theorem~\ref{thm:causalitylocal}, which holds for every stably causal Lorentzian manifold (since every such manifold admits a smooth temporal function by \cite{Min4}*{Thm.\ 4.100}).

\begin{proof}[Proof of Theorem~\ref{thm:causalitylocal}]
 Let $x \in M$. Since $J^+ \subseteq \nullR_\tau$ is always true, the $\subseteq$ inclusion is trivial. In order to show $\supseteq$ we construct a suitable relatively compact open neighborhood $U$ of $x$ and construct causal curves locally similar to the proof of Theorem~\ref{thm:causalityglobal}. Only Step 2 made use of the complete Riemannian metric $g_R$ and Step 4 used global hyperbolicity of $(M,g)$ and have to be carried out slightly different.
 
 By local compactness of $M$ there is an $r>0$ sufficiently small such that $\hat B^\tau_{3r}(x)$ is relatively compact (the open ball with respect to $\nulld_\tau$ of radius $3r$). Consider $U := \hat B^\tau_{r}(x)$. Suppose now $p,q \in U$ and $(p,q) \in \nullR_\tau$, i.e.,
 \[
  \nulld_\tau(p,q) = \tau(q) - \tau(p) >0.
 \]
 By definition of $U$ and the triangle inequality also $\nulld_\tau(p,q) < 2r$. Let $(\beta_n)_n$ be a sequence of piecewise causal paths $\beta_n \colon [0,L_n] \to M$ that approximates $\nulld_\tau(p,q)$. Hence we may fix any $\varepsilon \in (0,\nulld_\tau(p,q))$ and assume without loss of generality that for all $n$,
 \begin{equation} \label{eq:nullLepsilon}
 \nullL_\tau(\beta_n) < \nulld_\tau(p,q) + \varepsilon < 2r.
 \end{equation}
In particular, all $\beta_n$ are contained in $\hat B^\tau_{3r}(x)$.
 
 Since $\tau$ is assumed to be temporal the Wick-rotated metric $g_W$ of $g$ exists (see Section \ref{ssec:wick}). We assume that the $\beta_n$ are parametrized by $g_W$-arclength, and therefore $L_n = L_W(\beta_n)$. By Lemma \ref{lem:wick} and \eqref{eq:nullLepsilon} we then obtain the following estimates:
 \begin{align}\label{eq:Ln}
  \nulld_\tau(p,q) \leq \nullL_\tau(\beta_n) \leq L_n \leq \sqrt{2} \nullL_\tau (\beta_n) \leq \sqrt{2} \left( \nulld_\tau(p,q) + \varepsilon \right).
 \end{align}
 It is more convenient to have all the $\beta_n$ defined on the same interval $[0,L]$, which we achieve by extending each $\beta_n$ as follows: Set $L := \sqrt{2} (\nulld_\tau(p,q) + \varepsilon)$. Then attach to each $\beta_n$ a future-directed causal curve $\tilde{\beta}_n$ starting at $q$ of $g_W$-length $\tilde{L}_n \leq L - L_n$. Notice that
 \[\tilde{L}_n \leq L - L_n < (\sqrt{2} -1)\nulld_\tau(p,q)+ \sqrt{2}\varepsilon < 2 (\sqrt{2} -1) r+ \sqrt{2}\varepsilon  < r \]
 for $\varepsilon$ small enough. Since $\nullL_\tau(\tilde{\beta}_n) \leq \tilde{L}_n$  (by Lemma \ref{lem:wick}) and $\tilde{\beta}_n$ starts at $q \in \hat{B}^\tau_r(x)$, it follows that $\tilde{\beta}_n$ is contained in $\hat B^\tau_{3r}(x)$. This also proves that indeed a long enough extension up to $\tilde{L}_n = L - L_n$ exists, given that $\hat B^\tau_{3r}(x)$ is relatively compact and the spacetime $(M,g)$ is non-totally imprisoning (because it admits a time function). We have thus obtained a sequence of piecewise causal curves, denoted again as $\beta_n \colon [0,L] \to \hat B^\tau_{3r}(x)$. The curves $\beta_n$ start at $p = \beta_n(0)$, reach $q = \beta_n(L_n)$, and then continue to their endpoint $\beta_n(L)$. The local estimate \eqref{eq:locminimizing} of Step 1 holds, too, because the extension is causal: If $s\leq t \leq L_n$ it follows as in the proof of Theorem \ref{thm:causalityglobal}. If $s \leq L_n \leq t \leq L$ it follows from \eqref{eq:locminimizing} on $[0,L_n]$ and future causality on $[L_n,L]$ via
 \begin{align*}
  \vert \tau(\beta_n(t)) - \tau(\beta_n(s)) \vert
   &\leq  \nullL_\tau(\beta_n|_{[s,t]}) = \nullL_\tau(\beta_n|_{[s,L_n]}) + \nullL_\tau(\beta_n|_{[L_n,t]})  \\ &= \nullL_\tau(\beta_n|_{[s,L_n]}) + \tau(\beta_n(t)) - \tau(\beta_n(L_n)) \\ &\leq \tau(\beta_n(t)) - \tau(\beta_n(s)) + \frac{1}{n}.
 \end{align*}
 If $L_n \leq s\leq t \leq L$ the inequality \eqref{eq:locminimizing} follows directly from the future causality of $\beta_n|_{[L_n,L]}$ even without error term $\frac{1}{n}$.
 
 Let $B := \hat B^\tau_{3r}(x)$. By the Arzel\`{a}--Ascoli Theorem applied on the compact metric space $(\overline{B},d_W|_{\overline{B}})$, there exists a subsequence, again denoted by $(\beta_n)_n$, that uniformly converges to a $1$-Lipschitz continuous limit curve $\beta \colon [0,L] \to \overline{B}$ as in Step 2 of the proof of Theorem~\ref{thm:causalitylocal}. 
 
 Step 3 in the proof of Theorem~\ref{thm:causalitylocal} can be used verbatim with $g_R = g_W$ (and $\Omega \equiv 1$) since completeness of $g_R$ was not needed here. Thus $\beta$ is future-directed causal.
 
 It remains to prove that $\beta$ reaches the point $q$, which is now easier thanks to the fact that $(\beta_n)_n$ converges to $\beta$ uniformly on $[0,L]$. Recall that $\beta_n(L_n) =q$ for all $n \in \nat$. By passing to a subsequence, if necessary, we may assume that $L_n \to L_\infty \leq L$. But then, given $\delta > 0$, for all $n$ large enough we obtain
 \begin{align*}
  d_W(& \beta(L_\infty),q) \\
  &\leq  d_W(\beta(L_\infty),\beta_n(L_\infty)) + d_W(\beta_n(L_\infty),\beta_n(L_n)) + d_W(\beta_n(L_n),q) \\ &\leq \delta + |L_\infty -L_n| + 0 < 2\delta,
 \end{align*}
 and therefore conclude that $\beta(L_\infty) = q$.

 Thus $\beta$ is indeed a future-directed causal curve from $p$ through $q$ and hence $(p,q) \in J^+$.
\end{proof}


\subsection{Alternative approaches and proof of Theorem~\ref{thm:admcausalityglobal}}
\label{ssec:admcausalityglobal}

In earlier work of Sormani and Vega the important class of warped product spacetimes $I \times_f \Sigma$ with interval $I \subseteq \real$ and complete Riemannian fibers $\Sigma$ was already shown to encode causality for certain temporal functions \cite{SoVe}*{Thm.\ 3.25}. It is also easy to see that $\nulld_\tau$ encodes causality if all null distances are realized by piecewise causal curves and $\tau$ is locally anti-Lipschitz \cite{AlBu}*{Rem.\ 3.22}. It remained an open problem to understand causality encoding in the general case. Independently to our approach, Sakovich and Sormani~\cite{SaSo} very recently obtained some results that are comparable to Theorem~\ref{thm:causalityglobal} and Theorem~\ref{thm:causalitylocal}. We briefly discuss their setting and how it compares to ours.
 
The global causality encoding result \cite{SaSo}*{Thm.\ 4.1} of Sakovich and Sormani is formulated for spacetimes with \emph{proper} locally anti-Lipschitz time functions, requiring that all time slabs $\tau^{-1}([\tau_1,\tau_2])$ with $[\tau_1,\tau_2] \subseteq \tau(M)$ are compact. The following argument shows that the level sets of a proper time function $\tau$ are (compact) \emph{Cauchy} hypersurfaces and thus the spacetimes that Sakovich and Sormani consider are, in particular, \emph{globally hyperbolic}: Suppose, for the sake of contradiction, that $\gamma$ is an inextendible causal curve on $M$ that does not intersect some $\tau$-level set. Without loss of generality, suppose that $[0,1] \in \tau(M)$ and that $\gamma$ intersects $\{\tau=0\}$ but not $\{\tau=1\}$. Then the piece of $\gamma$ lying in the compact set $\tau^{-1}([0,1])$ is future inextendible, contradicting the fact that any spacetime with a time function is non-totally imprisoning. Hence $\gamma$ must intersect every level set of $\tau$.
 
Therefore, our global Theorem~\ref{thm:causalityglobal} is applicable to a wider class of spacetimes (also those having noncompact Cauchy surfaces) while the result \cite{SaSo}*{Thm.\ 4.1} of Sakovich and Sormani is applicable to a wider class of time functions (locally anti-Lipschitz instead of temporal).
 
In Section~\ref{ssec:causalitycount} we show that the assumption of Cauchy level sets can, in general, not be relaxed (see Remark~\ref{rem:fcauchy} for a mild trivial extension). An example of Sakovich and Sormani \cite{SaSo}*{Ex.\ 2.2} that was constructed to show that noncompact level sets are problematic, in fact also already fails on a much more fundamental level because the spacetime is not globally hyperbolic.

Both proofs, that of Theorem~\ref{thm:causalityglobal} and that of Sakovich and Sormani~\cite{SaSo}*{Thm.~4.1}, rely on constructing a limit of a $\nulld_\tau$-minimizing sequence of piecewise causal curves, and showing that the limit is a (continuous) causal curve. Both proofs do this via the Arzel\`{a}--Ascoli Theorem, but while Sakovich and Sormani apply it using the null distance, we employ a complete Riemannian metric $g_R$ (directly related to $g$ via Wick-rotation, hence requiring $\tau$ temporal). This added regularity in our proof allows us to indeed obtain a locally Lipschitz (with respect to any Riemannian metric) limit curve $\beta$ and compute $g(\dot\beta,\dot\beta)$ explicitly. Sakovich and Sormani can work with locally anti-Lipschitz time functions by using special coordinate systems and do not rely on the regularity of $\beta$ to show that points are causally related. The important step to prove that the limit curve indeed reaches the desired endpoint is achieved by Sakovich and Sormani by the properness of $\tau$ (which implies that the whole sequence lies in a compact set) while we employ Cauchyness of the level sets (which can be noncompact, placing less restrictions on the spacetime, as discussed above).

In Theorem~\ref{thm:admcausalityglobal} we combine both approaches to obtain causality encodement for all locally anti-Lipschitz time functions with (future/past) Cauchy level sets (neither required to be proper nor temporal). See Example \ref{ex:Milne} for a physically relevant case where our result applies.

\begin{proof}[Proof of Theorem~\ref{thm:admcausalityglobal}]
 Suppose the level sets of $\tau$ are past Cauchy. We use the notation of Theorem~\ref{thm:causalityglobal} and sequence $(\beta_n)_n$ satisfying \eqref{eq:globminimizing}. Note that Step 1 in the proof of Theorem~\ref{thm:causalityglobal} does not require any specific property of $\tau$ either, so \eqref{eq:locminimizing} also holds. We can carry out Step 2 with respect to any complete Riemannian metric $h$ on $M$ (instead of $g_R$). Thus by the Arzel\`{a}--Ascoli Theorem we obtain a locally Lipschitz limit curve $\beta \colon [0,\infty) \to M$ from $p$.
 
 Since $d_h$ and $\nulld_\tau$ both induce the manifold topology, the uniform convergence with respect to $d_h$ on compact subintervals implies pointwise convergence $\beta_n(t)\to\beta(t)$ with respect to the null distance $\nulld_\tau$ for all $t \in [0,\infty)$ as $n\to\infty$. Since the induced length structure of $\nulld_\tau$, i.e.,
 \[ L_{\nulld_\tau}(\gamma) = \sup \left\{ \sum_{i=1}^k \nulld_\tau (\gamma(s_i),\gamma(s_{i-1})) \, | \,  a=s_0 < s_1 < \ldots < s_k =b \right\} \]
 for rectifiable paths $\gamma \colon [a,b] \to M$, is lower semicontinuous \cite{BBI}*{Prop.\ 2.3.4} and agrees with $\nullL_\tau$ on the class of piecewise causal curves \cite{AlBu}*{Prop.\ 3.8} we obtain
 \begin{align*}
  L_{\nulld_\tau}(\beta|_{[s,t]}) &\leq \lim_{n\to\infty} L_{\nulld_\tau}(\beta_n|_{[s,t]}) = \lim_{n\to\infty} \nullL_\tau(\beta_n|_{[s,t]}).
 \end{align*}
 Together with property \eqref{eq:locminimizing} and the continuity of $\tau$ we thus have that
 \[
  L_{\nulld_\tau}(\beta|_{[s,t]}) \leq \tau(\beta(t)) - \tau(\beta(s)) \leq \nulld_\tau(\beta(s),\beta(t)) \leq L_{\nulld_\tau}(\beta|_{[s,t]}).
 \]
 Hence $\beta$ is not only a $\nulld_\tau$-minimizing curve but also satisfies for all $s,t \in [0,\infty)$
 \[
  L_{\nulld_\tau}(\beta|_{[s,t]}) = \tau(\beta(t)) - \tau(\beta(s)).
 \]
 As in the proof of Sakovich and Sormani~\cite{SaSo}*{Thm.\ 4.3} their local causality encoding property thus implies that $\beta$ is future-directed causal as continuous curve (in the sense of \cite{HaEl}*{p.\ 184}) starting at $p$. This completes Step 3.
 
 It remains to be shown that $\beta$ reaches $q$ (in fact, at this point it could still be constant $p$). We proceed in a similar fashion as in Step 4 of the proof of Theorem~\ref{thm:causalityglobal} (also taking into account Remark~\ref{rem:fcauchy}) and distinguish the two cases (i) $\tau(\beta(s_0)) = \tau(q)$ for some $s_0 \in[0,\infty)$ and (ii) $\tau(\beta(s)) < \tau(q)$ for all $s$:

 (i) extends verbatim (replacing $d_R$ by $d_h$) and implies that $\beta(s_0)=q$.
 
 (ii) requires us to prove that $b_n$ is uniformly bounded from above. Note that, again with respect to the arbitrary complete Riemannian metric $h$ chosen for the convergence in Step 2, by \cite{Bur}*{Thm.\ 4.11},
 \begin{align*}
  b_n &= L_h(\beta_n|_{[0,b_n]}) \\ &= \sup \left\{ \sum_{i=1}^k d_h(\beta_n(s_i),\beta_n(s_{i-1})) \, | \, 0=s_0 < s_1 < \ldots < s_k = b_n \right\}.
 \end{align*}
 Since $\tau$ is locally anti-Lipschitz, by Step 1 in the proof of Lemma~\ref{lem:localbiLip}, for every point in $M$ there exists a neighborhood $U$ and a constant $C>0$ such that
 \begin{equation} \label{eq:antiLipproof}
 d_h(x,y) \leq C \nulld_\tau(x,y)
 \end{equation}
 for all $x,y \in U$. By the (reverse) argument in the proof of Theorem~\ref{thm:nonsmooth} on page \pageref{proof:nonsmooth}, we can even assume that $C$ is such that \eqref{eq:antiLipproof} holds on the entire compact set $\overline{W}$. Since all $\beta_n(s_i) \in \overline{W}$ by construction, and again by \cite{AlBu}*{Prop.\ 3.8}, we have
 \begin{align*}
  b_n &\leq \sup \left\{ \sum_{i=1}^k C \nulld_\tau(\beta_n(s_i),\beta_n(s_{i-1})) \, | \, 0=s_0 < s_1 < \ldots < s_k = b_n \right\} \\
  &\leq C L_{\nulld_\tau}(\beta_n|_{[0,b_n]}) = C \nullL_\tau(\beta_n|_{[0,b_n]}) \leq C (\tau(q)-\tau(p)+1).
 \end{align*}
 Proceeding again as in the proof of Theorem~\ref{thm:causalityglobal} yields a contradiction.
 
 Therefore, $\beta$ is a locally Lipschitz future-directed causal curve from $p$ to $q$, and $ q \in J^+(p)$. 
\end{proof}

Theorem~\ref{thm:admcausalityglobal} (for future Cauchy level sets) allows us to immediately prove \emph{global} causality encodement for a large and physically relevant class of time functions for which \emph{local} causality encodement was already shown by Sakovich and Sormani~\cite{SaSo}*{Cor.\ 1.2}.

\begin{proof}[Proof of Corollary~\ref{cor:cosmo}]
 By Lemma~\ref{temporalareLaL} the regular cosmological time function $\tau \colon M \to (0,\infty)$ is weak temporal. By \cite{AGH}*{Prop.\ 2.6} the level sets of $\tau$ are future Cauchy, hence the result follows from Theorem~\ref{thm:admcausalityglobal}.
\end{proof}

We conclude with a basic cosmological example for which Theorem~\ref{thm:admcausalityglobal} is directly applicable.

\begin{ex}[Milne model]\label{ex:Milne}
 Recall that the $n+1$-dimensional Milne model $(M,g)$ is a globally hyperbolic spacetime which can be viewed as the chronological future $I^+(0)$ of the origin in the Minkowski spacetime $\real^{1,n}$ (see, for instance, \cite{Li}). The cosmological time function $\tau \colon M \to (0,\infty)$ is the Lorentzian distance $d_g$ from the origin, i.e.,
 \[
  \tau(p) := \sup_{q\in J^-(p)} d_g(q,p).
 \]
 By Lemma~\ref{temporalareLaL} $\tau$ is a weak temporal function and the level sets of $\tau$ are the noncompact hyperboloids (which are Cauchy). Thus by our Theorem~\ref{thm:admcausalityglobal} we see that $\nulld_{\tau}$ encodes causality globally. Due to noncompactness the result~\cite{SaSo}*{Thm.\ 4.1} is not applicable. Since, however, the Milne model can be viewed as a warped product when expressed in the right coordinates \cite{Li}*{Eq.\ 1.1}, causality encodement already follows from an earlier result of Sormani and Vega~\cite{SoVe}*{Thm.\ 3.25}.
\end{ex}

\begin{rem}[Local causality-encoding results]
One can extend the local causality-encoding Theorem~\ref{thm:causalitylocal} to locally anti-Lipschitz time functions along the same lines. Note that in this local result a neighborhood $U$ of $x$ is constructed on which \emph{any} two points $p, q \in U$ can be compared as in \eqref{eq:causenc}, while in the local result \cite{SaSo}*{Thm.\ 1.1} of Sakovich and Sormani the point $p = x$ is fixed and only $q$ can be chosen freely.
\end{rem}


\subsection{Counterexamples}
\label{ssec:causalitycount}

We conclude this section with a series of examples that show that the local Theorem~\ref{thm:causalitylocal} with respect to temporal functions cannot be promoted to a global statement in the spirit of Theorem~\ref{thm:causalityglobal}, even if the spacetime is globally hyperbolic, but the $\tau$-level sets are not Cauchy (Example~\ref{ex:counterexgh}).

\medskip
In order to better contextualize our examples, we also consider the $\K$ relation. Recall that $\K$ is defined as the (unique) smallest closed and transitive relation containing $\J$. By Lemma~\ref{lem:Rprop} we therefore know that \[ \J \subseteq \K \subseteq \nullR_\tau.\]
The definition of $\K$ is due to Sorkin and Woolgar~\cite{SoWo}.  Furthermore, a spacetime $(M,g)$ is called \emph{$\K$-causal} if the $\K$ relation is antisymmetric, a condition later shown to be equivalent to stable causality by Minguzzi~\cite{Min}, and hence also equivalent to the existence of time function \cites{Haw,Min2}.

\begin{ex}[$J^+ \subsetneq K^+ = \nullR_\tau$]
Allen and Burtscher constructed examples \cite{AlBu}*{Ex.\ 3.23, 3.24} by removing points and lines from Minkowski space for which causality is not encoded, i.e., $\J \neq \nullR_\tau$. Notably, $K$-causality is still encoded, meaning that $\K = \nullR_\tau$.
\end{ex}

Sormani and Vega gave another example \cite{SoVe}*{Prop.\ 3.4} for which one can check that $\J \subsetneq \K \subsetneq \nullR_\tau$. Their example is Minkowski space with the time function $\tau = t^3$, which is is not locally anti-Lipschitz, and the null distance is not definite. We modify said example to obtain a definite null distance for which neither $\J$ nor $\K$ are encoded.

\begin{ex}[$\J \subsetneq \K \subsetneq \nullR_\tau$] \label{counterexK}
Consider $M := \real^{1,1} \setminus \{(0,x) \mid x \geq 0 \}$ equipped with the usual Minkowski metric $g := -dt^2 + dx^2$. We define the function $\tau \colon M \to \real$ by
\begin{equation*}
    \tau(t,x) := \begin{cases} t^3 &\text{if } x > 0, \\ t^3 + t x^2 &\text{if } x \leq 0. \end{cases}
\end{equation*}

First observe that $\tau$ is a temporal function on $(M,g)$ because its gradient vector is timelike: For $x > 0$, this is trivial (note that then $t \neq 0$, by definition of $M$). For $x \leq 0$, since $(0,0) \not\in M$, it follows from
\begin{align*}
    g(\nabla \tau,\nabla\tau) &= -\left( \partial_t \tau \right)^2 + \left( \partial_x \tau \right)^2 \\
    &= -\left( 3t^2+x^2 \right)^2 + \left( 2tx \right)^2 \\
    &= -9t^4 -x^4 - 6 t^2 x^2 + 4 t^2 x^2 < 0.
\end{align*}

Next we show that every pair of points $p = (t_p,x_p)$, $q=(t_q,x_q)$ with $x_p,x_q > 0$ and $t_q < 0 < t_p$ satisfies
\begin{equation} \label{dppbar}
   \nulld_\tau(p,q) = \tau(p) - \tau(q),
\end{equation}
despite the fact that clearly not all such $p$ and $q$ are related by $\J$ or $\K$: Let $k \in \nat$ be large enough so that $t_p > x_p/k$ and $\vert t_q \vert > x_q/k$. Consider the points (see Figure \ref{fig:counterexK})
\begin{equation*}
    q_1 := (x_p/k,x_p), \quad q_2 := (x_p/k,0), \quad q_3 := (0,-\min\{x_p,x_q\}/k).
\end{equation*}

\begin{figure}
    \centering
    \begin{tikzpicture}
    \begin{axis}[xmin=-1,xmax=4,
                 ymin=-1.5,ymax=2,
                 axis equal image,
                 axis x line=middle,
                 axis y line=middle,
                 x axis line style={draw=none, insert path={(axis cs:-1,0) edge[-] (axis cs:0,0) (axis cs:0,0) edge[ dashed] (axis cs:4,0)}},
                 xlabel={$x$},
                 ylabel={$t$},
                 xtick=\empty,
                 ytick=\empty,
                 clip=false]
        \filldraw[black] (axis cs: 3,1.5) circle (2pt) node[anchor = west] {$p$};
        \filldraw[black] (axis cs: 3,0.25) circle (2pt) node[anchor = west] {$q_1$};
        \filldraw[black] (axis cs: 0,0.3) circle (2pt) node[anchor = south east] {$q_2$};
        \filldraw[black] (axis cs: -0.25,0) circle (2pt) node[anchor = south east] {$q_3$};
        \filldraw[black] (axis cs: 2.5,-1) circle (2pt) node[anchor = west] {$q$};
        \filldraw[draw=black,fill=white] (axis cs: 0,0) circle (2pt);
        \draw (axis cs: 3,1.5) -- (axis cs: 3,0.3);
        \draw foreach \k in {1, ..., 5} {(axis cs: {3.6-0.6*\k},0.3) -- (axis cs: {3.3-0.6*\k},0.6) -- (axis cs: {3-0.6*\k},0.3)};
        \draw (axis cs: 0,0.3) -- (axis cs: -0.25,0) -- (axis cs: 0,-0.25);
        \draw foreach \k in {1, ..., 5} {(axis cs: {3-0.5*\k},-0.25) -- (axis cs: {2.75-0.5*\k},-0.5) -- (axis cs: {2.5-0.5*\k},-0.25)};
        \draw (axis cs: 2.5,-1) -- (axis cs: 2.5,-0.25);
    \end{axis}
    \end{tikzpicture}
    \caption{A piecewise causal curve that approximates the distance between the points $p$ and $q$ in Example \ref{counterexK}, giving the upper bound \eqref{estimatedpr}.}
    \label{fig:counterexK}
\end{figure}
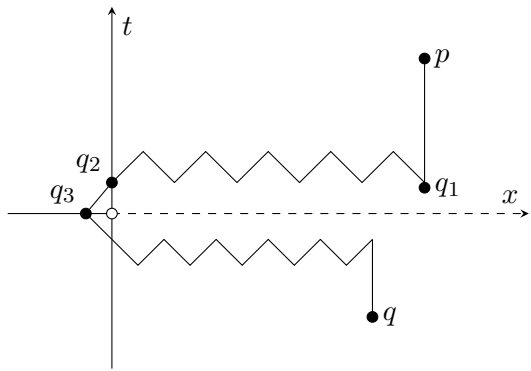

\noindent Since $p \in J^+(q_1)$ and $q_2 \in J^+(q_3)$ trivially
\begin{align} \label{dq0q1}
    &\nulld_\tau(p,q_1) = t_p^3 - \left(\frac{x_p}{k} \right)^3,  &\nulld_\tau(q_2,q_3) = \left( \frac{x_p}{k} \right)^3.
\end{align}
Moreover, taking a piecewise null curve from $q_1$ to $q_2$ that consists of $2k$ segments between $t=x_p/k$ and $t=2x_p/k$ (see Figure \ref{fig:counterexK}), we get that
\begin{equation} \label{dq1q2}
    \nulld_\tau(q_1,q_2) \leq 2k \left( \left( \frac{2 x_p}{k}  \right)^3 - \left(  \frac{x_p}{k} \right)^3 \right) = 14 \frac{x_p^3}{k^2}.
\end{equation}
Combining \eqref{dq0q1} and \eqref{dq1q2} with the triangle inequality proves
\begin{equation*}
  \nulld_\tau(p,q_3) \leq t_p^3 + 14 \frac{x_p^3}{k^2},
\end{equation*}
and by symmetry we obtain the analogous estimate for $\nulld_\tau(q,q_3)$, thus
\begin{equation} \label{estimatedpr}
    \nulld_\tau(p,q) \leq \nulld_\tau(p,q_3) + \nulld_\tau(q_3,q) \leq  t_p^3 + 14 \frac{x_p^3 + x_q^3}{k^2} + \left\vert t_q \right\vert^3,
\end{equation}
and taking the limit $k \to \infty$ implies \eqref{dppbar}, as desired.
\end{ex}

In the previous example, we have constructed a minimizing sequence of piecewise causal curves by choosing curves that are close to a ``barrier" (the removed positive $x$ axis). This barrier, however, also makes the spacetime causally discontinuous. Since the time function $\tau$ is perfectly regular (it is $C^1$ with timelike gradient, and can easily be smoothed out), one might suspect that causal discontinuity is the reason that causality is not encoded by $\nulld_\tau$ in Example~\ref{counterexK}. This motivates the next example, where we construct a causally simple spacetime with a temporal function $\tau$ but causality is still not encoded in $\nulld_\tau$. Recall that causal simplicity means that the causal relation $J^+$ is antisymmetric and closed, and sits only one step below global hyperbolicity on the causal ladder. In order to achieve this effect, instead of approaching a barrier, we construct an example with a minimizing sequence of piecewise null curves that runs off to infinity.

\begin{ex}[$\J = \K \subsetneq \nullR_\tau$] \label{counterexsimple}
    Let $M:= \real^{3}$ with coordinates $(t,x,y)$ and warped product metric tensor
    \begin{equation} \label{metricJK}
        g := \cosh^2 (x) \left( -dt^2 + dy^2 \right) + dx^2.
    \end{equation}
    A change of coordinates $z:=\arctan\left(\sinh(x)\right)$ reveals that $g$ is conformal to the Minkowski metric $-dt^2+dy^2+dz^2$, where $z \in (- \pi/2, \pi/2)$. Hence $(M,g)$ is causally simple with $\J = \K$.

    Furthermore, note that all $\{x=x_0\}$ planes are conformal to Minkowski space $\real^{1,1}$, while each $\{ y = y_0 \}$ plane is isometric to the universal cover of $AdS^2$. Therefore, the induced null geodesics $s \mapsto (t_\pm(s),x_\pm(s))$ in the $\{ y = y_0\}$ plane going through a point $(t_0,0)$ are given by
    \begin{align}
        t_\pm(s) &= 2 \arctan \left( \tanh \left( \frac{s}{2} \right) \right) + t_0,  \label{nullgeo} \\
        x_\pm(s) &= \pm s,  \nonumber
    \end{align}
    the subscript $+$ or $-$ indicating the right- or left-going geodesics respectively \cite{Klaas}*{Sec.\ 5.10}. Moreover, the function
    \begin{equation*}
        \tau(t,x,y) := \cosh^{-1} (x) t + t^3,
    \end{equation*}
    is a steep temporal function on $(M,g)$ since
    \begin{align*}
        g(\nabla\tau,\nabla\tau) &= -\cosh^2(x) \left( \partial_t \tau \right)^2 + \left( \partial_x \tau \right)^2 \\
    &= -\cosh^2(x) \left( \cosh^{-1}(x) + 3t^2 \right)^2 + \left( -\cosh^{-1}(x) \tanh(x) t \right)^2 \\
    &= -1 - \cosh^2(x) 9t^4 - \left( 6 \cosh(x) - \cosh^{-2}(x) \tanh^2 (x) \right) t^2 \leq -1,
    \end{align*}
    where in the last line, we have used that $|\tanh(x)| < 1 \leq \cosh(x)$ for all $x \in \real$.
    
    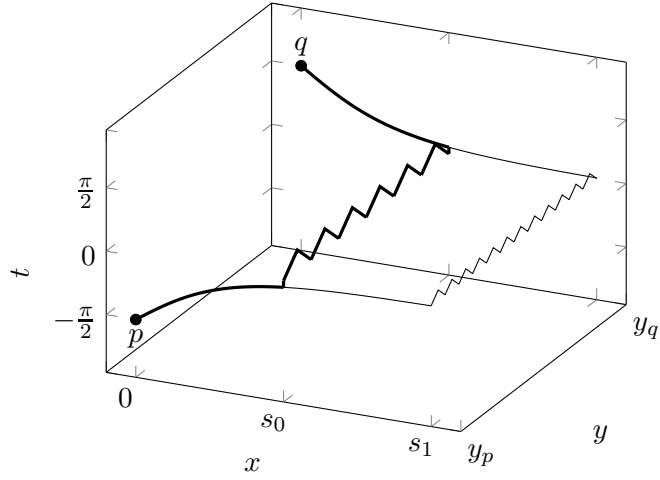
\begin{figure}
\centering
\begin{tikzpicture}
\begin{axis}
    [xlabel = {$x$},
    ylabel = {$y$},
    zlabel = {$t$},
    xmin=-0.5,xmax=5.5,
    ymin=0,ymax=6,
    zmin=-3,zmax=3,
    xtick={0,2.5,5},
    xticklabels={$0$,$s_0$,$s_1$},
    ytick={0,6},
    yticklabels={$y_p$,$y_q$},
    ztick={-1.5708,0,1.5708},
    zticklabels={$-\frac{\pi}{2}$,$0$,$\frac{\pi}{2}$},
    ]
    
\addplot3[
    domain=2.5:5,
    samples = 50,
    samples y=0,
    smooth,
]
({x},
{0},
{2*rad(atan(tanh(x*0.5)))-0.5*pi});

\addplot3[
    domain=2.5:5,
    samples = 50,
    samples y=0,
    smooth,
]
({x},
{6},
{2*rad(atan(tanh(-x*0.5)))+0.5*pi});

\draw[
	  ] (axis cs: 5,0,{2*rad(atan(tanh(2.5)))-0.5*pi}) -- (axis cs: 5,0,0);
\draw[
		    ] (axis cs: 5,6,{2*rad(atan(tanh(-2.5)))+0.5*pi}) -- (axis cs: 5,6,0);

\draw[
		    ] foreach \k in {1, ..., 12} {(axis cs: 5,0.5*\k-0.5,0) -- (axis cs: 5,0.5*\k-0.25,0.25) -- (axis cs: 5,0.5*\k,0)};
    
\addplot3[
    domain=0:2.5,
    samples = 50,
    samples y=0,
    smooth,
    very thick
]
({x},
{0},
{2*rad(atan(tanh(x*0.5)))-0.5*pi});

\addplot3[
    domain=0:2.5,
    samples = 50,
    samples y=0,
    smooth,
    very thick
]
({x},
{6},
{2*rad(atan(tanh(-x*0.5)))+0.5*pi});

\filldraw[black] (axis cs: 0,0,-0.5*pi) circle (2pt) node[anchor=north] {$p$};
\filldraw[black] (axis cs: 0,6,0.5*pi) circle (2pt) node[anchor=south] {$q$};

\draw[very thick] (axis cs: 2.5,0,{2*rad(atan(tanh(1.25)))-0.5*pi}) -- (axis cs: 2.5,0,0);
\draw[very thick] (axis cs: 2.5,6,{2*rad(atan(tanh(-1.25)))+0.5*pi}) -- (axis cs: 2.5,6,0);

\draw[very thick] foreach \k in {1, ..., 6} {(axis cs: 2.5,\k-1,0) -- (axis cs: 2.5,\k-0.5,0.5) -- (axis cs: 2.5,\k,0)};
\end{axis}
\end{tikzpicture}
\caption{Two piecewise causal curves going trough the points $p_{s_0}$ and $q_{s_0}$ (thick) and $p_{s_1}$ and $q_{s_1}$ (thin), respectively, that approximate the null distance between $p$ and $q$ in Example \ref{counterexsimple}. The thin curve yields a better approximation.}
\label{fig:counterexsimple}
\end{figure}

    Consider now two points $p,q \in M$ of the form
    \begin{align*}
        &p := \left(-\frac{\pi}{2},\ 0,\ y_p\right), &q := \left(\frac{\pi}{2},\ 0,\ y_q \right),
    \end{align*}
    where $\vert y_p - y_q \vert > \pi$. We are going to show that $(p,q) \in \nullR_\tau$, i.e., that
    \begin{equation} \label{dpq2}
        \nulld_\tau (p,q) = \tau(q) - \tau(p),
    \end{equation}
    despite the fact that $(p,q) \not\in \J$. To see that $(p,q) \not\in \J$, note that by definition of $g$ in \eqref{metricJK} the $\{x = 0\}$ plane is isometric to $1+1$-Minkowski space, and that the projection $(t,x,y)\mapsto (t,y)$ of any causal curve from $p$ to $q$ in $M$ to $\{x = 0\}$ remains causal in $\{ x = 0 \}$. The condition $\vert y_p - y_q \vert > \pi$ then implies that $(p,q) \not\in \J_{\{x = 0\}}$, and therefore also $(p,q) \not\in \J$.
    
    In order to show \eqref{dpq2}, we proceed in three steps. First, define
    \begin{align*}
        &p_s = (0,s,y_p), &q_s=(0,s,y_q),
    \end{align*}
    where $s \in \real$ is arbitrary. By \eqref{nullgeo}, a null geodesic in the $\{y=y_p\}$ plane starting at $p$ eventually reaches the point $(t_+(s),s)$ where $t_+(s) < 0$. It follows that $(p,p_s) \in \J$, and by the time reversed argument on the $\{y=y_q\}$ plane, that $(q_s,q) \in \J$, altogether implying that
    \begin{align} \label{dpps}
        &\nulld_\tau(p,p_s) = \tau(p_s) - \tau(p) = - \tau(p), &\nulld_\tau(q_s,q) = \tau(q) - \tau(q_s) = \tau(q).
    \end{align}
    The second step is to estimate the null distance between $p_s$ and $q_s$. Given that the $\{x = s\}$ plane is conformal to Minkowski space and the null distance is conformally invariant it is easy to construct piecewise causal curves between $p_s$ and $q_s$. The null distance induced on each plane is different though, and since $\tau|_{\{x=x_s\}} \to t^3$ as $s\to\infty$ we have at the ``boundary" of our spacetime an indefinite null distance that cannot distinguish any points in the $\{t=0\}$ slice \cite{SoVe}*{Prop.\ 3.4}. We make this intuitive picture precise by constructing in each $\{x=x_s\}$ plane a piecewise null curve $\beta_{s,k}$ that bounces $k$ times between $t=0$ and $t=\vert y_p - y_q \vert/k$ (see Figure \ref{fig:counterexsimple}). Using the curves $\beta_{s,k}$ to estimate the null distance, we obtain the upper bound
    \begin{align}
        \nulld_\tau(p_s,q_s) &\leq \liminf_{k\to\infty}\nullL_\tau(\beta_{s,k}) \nonumber \\ &= \liminf_{k\to\infty} 2k \left(\frac{\left\vert y_p - y_q \right\vert^3}{k^3} + \cosh^{-1}(s)\frac{\left\vert y_p - y_q \right\vert}{k} \right) \nonumber \\
        &= 2\cosh^{-1}(s)\left\vert y_p - y_q \right\vert.
        \label{nullLsigma}
    \end{align}
    Finally, the triangle inequality together with \eqref{dpps} and \eqref{nullLsigma} yields
    \begin{align*}
        \nulld_\tau(p,q) &\leq \lim_{s \to \infty} \left( \nulld_\tau(p,p_s) + \nulld_\tau(p_s,q_s) + \nulld_\tau(q_s,q) \right)
        = \tau(q) - \tau(p).
    \end{align*}
    This finishes the proof of \eqref{dpq2}, since the opposite inequality is always true.
    
Note that in our proof, $|y_p - y_q|$ can be chosen arbitrarily large while $\nulld_\tau(p,q) = \pi+ \frac{\pi^3}{4}$ remains the same. Therefore the $\nulld_\tau$-ball at $p$ of radius $R > \pi + \frac{\pi^3}{4}$ is unbounded with respect to the usual Euclidean distance.
\end{ex}

Given that in the causal ladder of spacetimes, causal simplicity comes just before global hyperbolicity, the previous Example~\ref{counterexsimple} shows that the assumptions in Theorem~\ref{thm:causalityglobal} are sharp in view of the causal structure required. The following and final example shows that even on a globally hyperbolic spacetime, Cauchyness of the time function cannot simply be dropped in Theorem~\ref{thm:admcausalityglobal}. This is not surprising, because non-Cauchy temporal functions on globally hyperbolic spacetimes can have a much wilder behavior than their Cauchy counterparts, such as topology changes of the level sets \cite{San2}.

\begin{ex}[$J^+ = K^+ \subsetneq \nullR_\tau$ for non-Cauchy locally anti-Lipschitz time function in globally hyperbolic spacetime] \label{ex:counterexgh}
We show that the (future/past) Cauchy assumption in Theorem~\ref{thm:admcausalityglobal} cannot be relaxed. To this end we construct an example that combines aspects of Examples~\ref{counterexK} and \ref{counterexsimple} in the sense that a $\nulld_\tau$-minimizing sequence of piecewise causal curves between certain points approaches both a barrier (in the $x$ direction) \emph{and} runs off to infinity (in the $y$ direction).
 
 The spacetime under consideration is
 \begin{align*}
 M := \, & \{(t,x,y) \mid t > 0, x > t-1\} \\
 &\cup \{(t,x,y) \mid t < 0, x < t+1\} \\ 
 &\cup \{ (0,x,y) \mid -1 < x < 1 \}  \subseteq \real^{1,2},
 \end{align*}
 considered as subset of the $(2+1)$-dimensional Minkowski space with metric $g := -dt^2 + dx^2 + dy^2$ (see Figure~\ref{fig:counterexgh}). Clearly, $M$ is globally hyperbolic. We equip $M$ with the continuous function
 \[
 \tau(t,x,y) := t^3 + \Psi(t,x)\cosh^{-1}\left(\frac{y}{2}\right),
 \]
 where
 \[
 \Psi(t,x) := \begin{cases} \sqrt{(t+1)^2-x^2} &\text{if } \vert t+1 \vert > \vert x \vert, \\ 0 & \text{otherwise.} \end{cases}
 \]
 We show that (i) $\tau$ is a time function for $(M,g)$ and (ii) that the corresponding null distance $\nulld_\tau$ does not encode causality globally. Theorem~\ref{thm:admcausalityglobal} thus implies that $\tau$ is not (future/past) Cauchy, as can also be seen by considering causal curves in Minkowski spacetime which leave the region $M$.

 \begin{figure}
\includegraphics{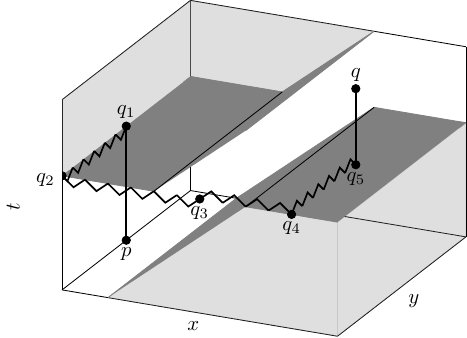}
\caption{A piecewise causal curve that approximates the null distance between $p$ and $q$ in Example~\ref{ex:counterexgh}. Only the space in between the two grey surfaces is part of the spacetime. The curve comes $\varepsilon$-close to the boundary, but remains entirely within $M$.}
\label{fig:counterexgh}
\end{figure}

(i) We show that the gradient vector field $\nabla \tau$ is timelike almost everywhere, which is a sufficient condition for being a time function. In the region where $\Psi=0$, our function is simply $t^3$, and since said region does not include $\{t=0\}$, we have that $\nabla \tau$ is timelike there. It remains to consider the region where $\Psi \neq 0$. There the gradient of $\tau$ is given by
 \begin{align*}
     \nabla \tau = &-\left(3t^2 + \frac{t+1}{\sqrt{(t+1)^2-x^2}} \cosh^{-1}\left(\frac{y}{2}\right)\right) \partial_t \\[1em] &-\frac{x}{\sqrt{(t+1)^2-x^2}} \cosh^{-1}\left(\frac{y}{2}\right) \partial_x \\[1em] &-\frac{1}{2}\sqrt{(t+1)^2-x^2} \cosh^{-1}\left(\frac{y}{2}\right) \tanh\left(\frac{y}{2}\right) \partial_y,
 \end{align*}
 and therefore its norm is
 \begin{align*}
     g(\nabla \tau, \nabla \tau) = &-9 t^4 - 3 t^2 \frac{t+1}{\sqrt{(t+1)^2-x^2}} \cosh^{-1}\left(\frac{y}{2}\right)  \\[1em] &- \cosh^{-2}\left(\frac{y}{2}\right)+ \frac{1}{4}\Big( (t+1)^2-x^2\Big) \cosh^{-2}\left(\frac{y}{2}\right) \tanh^{2}\left(\frac{y}{2}\right).
 \end{align*}
 On the RHS, the terms on the first line are always negative (since $t+1>0$ in the region we are considering). If $t\leq 1$, then the second line is also negative, since then $(t+1)^2-x^2 \leq 4$ and $\vert \tanh(z) \vert <1$. If, on the other hand, $t \geq 1$, then the $9 t^4$ term dominates the whole expression (since also $\cosh(z)>1$). In either case, we have shown that $g(\nabla \tau, \nabla \tau)<0$, as desired.
 
(ii) It remains to be shown that causality is not encoded in the null distance. Concretely, we show that for all $p=(t_p,x_p,y_p)$ and $q=(t_q,x_q,y_q)$ with $t_p < -2 < 2 < t_q$ and $x_p < t_p + 1$, $x_q > t_q + 1$, 
 \[ \nulld_\tau(p,q) = \tau(q) - \tau(p),\]
 despite the fact that clearly not every two such points are causally related. The argument is depicted in Figure~\ref{fig:counterexgh}, and we omit some computations that are analogous to the ones in the previous examples. Choose $\varepsilon>0$ and follow a causal segment from $p$ to $q_1 = (-\varepsilon,x_p,y_p)$, so that
 \[
 \nulld_\tau(p,q_1) = \tau(p) - \varepsilon^3.
 \]
 Note that $q_1$ lies in a region where $\tau = t^3$. Therefore, given any (arbitrarily large) $R>0$, for $q_2 = (-\varepsilon,x_p,R)$ we have \[\nulld_\tau(q_1,q_2) \to 0 \text{ as }\varepsilon \to 0,\] similarly to the situation in Example~\ref{counterexK}. Next, let $q_3 = (0,0,R)$. Then \[ \nulld_\tau(q_2,q_3) \sim \varepsilon \text{ for $R$ such that } \cosh^{-1}(R) \leq \varepsilon^2, \]
 similar to what happens in Example~\ref{counterexsimple}. Finally, do a similar procedure backwards to get from $q_3$ to $q_4 = (\varepsilon,x_q,R)$ to $q_5 = (\varepsilon,x_q,y_q)$ (with arbitrarily small length) and then to $q$ (with length $\tau(q)-\varepsilon^3$).
 
 In conclusion, by choosing $R(\varepsilon)$ such that $\cosh^{-1}(R(\varepsilon)) \leq \varepsilon^2$ (as required above), we have that
 \[\lim_{\varepsilon \to 0} \sum_{i=1}^4 \nulld_\tau(q_i,q_{i+1}) = 0,\]
 and by the triangle inequality
 \[\nulld_\tau(p,q) \leq \lim_{\varepsilon \to 0} \left[ \nulld_\tau(p,q_1) + \sum_{i=1}^4 \nulld_\tau(q_i,q_{i+1}) + \nulld_\tau(q_5,q) \right] = \tau(q) - \tau(p),\]
 as claimed (the opposite inequality always holds).
\end{ex}

We end this section noting that the temporal function in Example~\ref{counterexsimple} is \emph{steep}, a notion already discussed at the beginning of Section~\ref{ssec:wick}. Since any temporal function is steep for a conformal transformation of $g$ (which leaves the null distance invariant)\, steepness is unrelated to causality encodement. The situation is different for completely uniform temporal functions (also called $h$-steep), because they are special Cauchy temporal functions. We define and make use of them in the following section.


\section{Completeness}\label{sec:complete}

In this final section we prove our Main Theorem~\ref{mainthm} which characterizes global hyperbolicity of $(M,g)$ by metric completeness of $(M,\nulld_\tau)$. Completeness is a global property, therefore  we cannot expect $(M,\nulld_\tau)$ to be complete for all choices of $\tau$ (even though they are locally equivalent by Section~\ref{sec:equiv}). In Section~\ref{sec:causality} we have observed that locally anti-Lipschitz Cauchy functions encode causality globally. Therefore, it comes at no surprise that this is a necessary ingredient for a completeness result. We use and generalize the following class of time functions recently introduced by Bernard and Suhr~\cites{BeSu,BeSu2} to study closed cone fields.

\begin{defn} \label{defn:hsteep}
 Let $(M,g)$ be a spacetime. A smooth function $\tau \colon M \to \real$ is called a \emph{completely uniform temporal function} if there exists a complete Riemannian metric $h$ on $M$ such that for all causal vectors $v \in TM$
 \begin{align}\label{hsteep}
  d\tau(v) \geq \|v\|_h.
 \end{align}
 We call $\tau$ a \emph{completely uniform weak temporal function} if it is weak temporal and \eqref{hsteep} holds almost everywhere.
\end{defn}

Originally these functions were called steep with respect to a (complete) Riemannian metric in \cite{BeSu}*{p.\ 473} and later renamed in \cite{BeSu2}*{Def.\ 1.2}. Subsequently, $f$-steep functions with respect to any positive homogeneous $C^1$ function (not just $f = \|.\|_h$ for $h$ a complete Riemannian metric) were also used by Minguzzi~\cite{Min3}*{p.\ 2} in the analysis of Lorentz--Finsler spaces.

It was shown by Bernard and Suhr~\cite{BeSu}*{Thm.\ 3} and later also by Minguzzi~\cite{Min3}*{Thm.\ 3.1} that the existence of completely uniform temporal function is equivalent to global hyperbolicity of the spacetime.

These results are key in the following refined version of Theorem~\ref{mainthm}.

\begin{thm}\label{thm1}
 Let $(M,g)$ be a spacetime.
 \begin{enumerate}
  \item If $\tau$ is a time function such that $(M,\nulld_\tau)$ is a complete metric space, then $\tau$ is a Cauchy time function. In particular, $(M,g)$ is globally hyperbolic.
  \item If $(M,g)$ is globally hyperbolic then there exists a completely uniform weak temporal function $\tau$, and for every such $\tau$, $(M,\nulld_\tau)$ is a complete metric space.
 \end{enumerate}
\end{thm}

Theorem~\ref{mainthm} is a direct corollary of Theorem~\ref{thm1}.

\begin{proof}
 (i) Assume that $(M,\nulld_\tau)$ is complete but $\tau$ is not a Cauchy time function. Then there exists, without loss of generality, a future-directed future-inextendible causal curve $\gamma \colon \real \to M$ such that $\lim_{s \to \infty} \tau(\gamma(s)) < \infty$. Consider the sequence $(p_n)_n$ of points given by $p_n = \gamma(n)$. Since the $p_n$ are causally related among each other $\nulld_\tau(p_n, p_m) = \vert \tau(p_n) - \tau(p_m) \vert$. Then the fact that $\tau\circ\gamma \colon \real \to \real$ is strictly increasing and bounded from above implies that $(p_n)_n$ is a Cauchy sequence in $(M,\nulld_\tau)$. By completeness there exists a limit point $p$, and since $\gamma$ is continuous, $p \in \overline{\gamma}$, a contradiction to the inextendibility of $\gamma$. Hence $\tau$ must be a Cauchy time function, and $(M,g)$ globally hyperbolic by Theorem~\ref{thm:Cauchytime}.
 
 \medskip
 (ii) By \cites{BeSu,Min3} $(M,g)$ is globally hyperbolic if and only if there is a completely uniform temporal function $\tau$ which with respect to a complete Riemannian metric $h$ satisfies \eqref{hsteep}. We show that any such (even only weak temporal) $\tau$ is anti-Lipschitz with respect to the (complete) distance $d_h$ induced by $h$, i.e., there is a $C>0$ such that for all $p,q \in M$
 \begin{align}\label{antilip}
  (p,q) \in J^+ \implies \tau(q) - \tau(p) \geq C d_h(p,q).
 \end{align}
 Pick any $q \in J^+(p)$ and $\gamma \colon [0,1] \to M$ a causal curve with $\gamma(0)=p$, $\gamma(1)=q$. Then by \eqref{hsteep}
 \begin{align*}
  \tau(q) - \tau(p) &= \int_0^1 d\tau\left( \dot{\gamma}(s) \right) ds \\
  &\geq \int_0^1 \sqrt{h\left(\dot{\gamma}(s),\dot{\gamma}(s)\right)} ds = L_h(\gamma) \geq d_h (p,q).
 \end{align*}
 Thus \eqref{antilip} holds globally, and a theorem of Allen and Burtscher \cite{AlBu}*{Thm.\ 1.6} implies that $(M,\nulld_\tau)$ is complete (and definite).
\end{proof}

\begin{rem}
Recall that $(M,\nulld_\tau)$ is always a locally compact length-metric space \cite{AlBu}*{Thm.\ 1.1}. If $(M,\nulld_\tau)$ is also complete, the Hopf--Rinow--Cohn-Vossen Theorem implies that any pair of points can be joined by a $\nulld_\tau$-length minimizing curve. Beware that the minimizer is, in general, only $\nulld_\tau$-rectifiable, but not necessarily piecewise causal \cite{AlBu}*{Ex.\ 3.17}. If $\tau$ is besides Cauchy also locally anti-Lipschitz, thanks to Theorem~\ref{thm:admcausalityglobal}, we still do know that the null distance between two points is their difference in time precisely when there is a causal curve between them.
\end{rem}

Applying Theorem~\ref{thm1}(ii) and then (i) proves the following result, originally shown by Bernard and Suhr \cite{BeSu}*{Thm.~3} for temporal functions (see also \cite{BeSu2}*{Lemma~1.3}).

\begin{cor}
 If a weak temporal function $\tau$ is completely uniform, then $\tau$ is Cauchy. \qed
\end{cor}

Since the cosmological time function does not attain negative values it is not Cauchy, and hence by Theorem~\ref{thm1}(i) the corresponding null distance is not complete. We conclude our paper with a counterexample that shows that non-completely uniform Cauchy temporal functions on globally hyperbolic spacetimes do, in general, also not imply metric completeness.

\begin{ex}[Cauchy temporal function with incomplete null distance] \label{ex:notcomplete}
In \cite{San}*{Sec.\ 6.4}, S\'anchez constructs a globally hyperbolic spacetime $(M,g)=(\real^2,-dt^2+f^2(t,x)dx^2)$ with a certain piecewise defined $L^1$-function $f\colon M \to (0,\infty)$ and such that $t$ is a Cauchy temporal function, but the spacelike slice $\{t=0\}$ is geodesically complete as Riemannian manifold $(\real, f^2(0,x)dx^2)$. Let $(0,x)$, $(0,y)$ be two points on the $\{t=0\}$ slice. Then we can estimate their null distance by a sequence of piecewise null curves $\gamma_n(s) = (\gamma_n^t(s),s)$ satisfying $0 \leq \gamma_n^t(s) \leq \frac{1}{n}$. We obtain
\begin{equation*}
    \nulld_t ((0,x),(0,y)) \leq \nullL_t(\gamma_n) = \int^y_x \vert \dot\gamma_n^t(s) \vert ds = \int^y_x f(\gamma_n^t(s),s) ds.
\end{equation*}
Applying dominated convergence to the right hand side yields
\begin{align*}
 \lim_{n \to \infty} \int^y_x f(\gamma_n^t(s),s) ds = \int^y_x f(0,s) ds \leq \|f(0,\cdot)\|_{L^1(\real)} < \infty.
\end{align*}
This implies that the sequence $(n)_n$ is Cauchy because for any $\varepsilon>0$, assuming that $m\leq n$ sufficiently large,
\[
 \nulld_t((0,m),(0, n)) \leq \int_m^n f(0,s) ds \leq \int_m^\infty f(0,s) ds < \varepsilon.
\]
The hypothetical limit point at $\infty$, however, is not in $M$. Therefore, $(M,\nulld_t)$ is incomplete.
\end{ex}

\section*{Statements and declarations}

Data sharing not applicable to this article as no datasets were generated or analysed during the current study.

This version of the article has been accepted for publication, after peer review, but is not the Version of Record and does not reflect post-acceptance improvements, or any corrections. The Version of Record is available online at: \texttt{http://dx.doi.org/10.1007/s00220-024-04936-5}.

\bibliographystyle{abbrv}
\bibliography{nulldistleo}

\end{document}